\documentclass[reqno, 10pt]{amsart}
\usepackage{amsmath}
\usepackage{upgreek}
\usepackage{amssymb}
\usepackage{mathrsfs}
\usepackage{amsthm}
\usepackage{multirow, multicol}

\usepackage{textcomp}
\usepackage{times}

\usepackage[scr=boondox]{mathalfa}

\usepackage{xcolor}
\usepackage{srcltx}
\usepackage{mathtools}
\usepackage{thmtools}
\makeatletter
\renewcommand\thmt@autorefsetup{%
  \@xa\def\csname\thmt@envname autorefname\@xa\endcsname\@xa{\thmt@thmname}%
}
\makeatother
\usepackage[colorlinks,allcolors=blue!75!black]{hyperref}
\usepackage{enumitem}

\usepackage[left=1in, right=1in, top=1.25in, bottom=1.25in]{geometry}

\allowdisplaybreaks

\newcommand{\bbC}{\mathbb{C}}
\newcommand{\bbH}{\mathbb{H}}
\newcommand{\bbL}{\mathbb{L}}
\newcommand{\bbV}{\mathbb{V}}
\newcommand{\bbW}{\mathbb{W}}

\def\bA{\mathbf{A}}
\def\bB{\mathbf{B}}
\def\bC{\mathbf{C}}
\def\bD{\mathbf{D}}

\def\bF{\mathbf{F}}
\def\bH{\mathbf{H}}
\def\bI{\mathbf{I}}
\def\bL{\mathbf{L}}

\def\bP{\mathbf{P}}
\def\bR{\mathbf{R}}
\def\bQ{\mathbf{Q}}
\def\bT{\mathbf{T}}
\def\bS{\mathbf{S}}
\def\bV{\mathbf{V}}
\def\bW{\mathbf{W}}

\newcommand{\G}{\Gamma}
\newcommand{\nablaG}{\nabla_\G}
\newcommand{\nablaM}{\nabla_M}
\newcommand{\divG}{\operatorname{div}_\G}
\newcommand{\divM}{\operatorname{div}_M}
\newcommand{\DeltaG}{\Delta_\G}
\newcommand{\DeltaM}{\Delta_M}
\newcommand{\DeltaS}{\Delta_S}

\def\be{\mathbf{e}}
\def\bg{\mathbf{g}}
\def\bh{\mathbf{h}}

\def\bu{\mathbf{u}}
\def\bv{\mathbf{v}}
\def\bw{\mathbf{w}}
\def\bx{\mathbf{x}}

\def\bchi{\boldsymbol{\chi}}

\def\bvarphi{\boldsymbol{\varphi}}
\def\bsigma{\boldsymbol{\sigma}}
\def\bnu{\boldsymbol{\nu}}

\def\btau{\boldsymbol{\tau}}

\def\bPsi{\boldsymbol{\Psi}}
\def\bSigma{\boldsymbol{\Sigma}}

\def\bun{\bu^{(n)}}
\def\bvi{\bv^{(i)}}
\def\bvn{\bv^{(n)}}
\def\bvk{\bv^{(k)}}

\def\bwN{\bw^{(N)}}
\def\bAi{\bA^{(i)}}

\def\bAn{\bA^{(n)}}
\def\bAl{\bA^{(l)}}
\def\bQn{\bQ^{(n)}}
\def\bPsiN{\bPsi^{(N)}}

\def\wtbAn{\wt\bA^{(n)}}
\def\wtbAm{\wt\bA^{(m)}}

\def\sP{\mathscr{P}}

\renewcommand{\bf}{\mathbf{f}}

\newcommand{\wt}{\widetilde}

\newcommand{\BB}{\mathrm{B}}
\newcommand{\CC}{\mathrm{C}}
\newcommand{\HH}{\mathrm{H}}
\newcommand{\LL}{\mathrm{L}}
\newcommand{\NN}{\mathrm{N}}
\newcommand{\RR}{\mathrm{R}}
\renewcommand{\SS}{\mathrm{S}}
\newcommand{\VV}{\mathrm{V}}
\newcommand{\WW}{\mathrm{W}}

\newcommand{\cB}{\mathcal{B}}

\newcommand{\cP}{\mathcal{P}}
\newcommand{\cQ}{\mathcal{Q}}
\newcommand{\cU}{\mathcal{U}}
\newcommand{\cV}{\mathcal{V}}
\newcommand{\cW}{\mathcal{W}}

\newcommand{\dd}[1][y]{\if#1y\,\fi{\mathrm d}}

\DeclareMathOperator{\vspan}{span}

\renewcommand{\div}{\operatorname{div}}

\newcommand{\tr}{\operatorname*{tr}}

\newcommand{\ELdG}{E_{\mathrm{LdG}}}

\theoremstyle{plain}

\newtheorem{thm}{Theorem}[section]
\newtheorem{proposition}[thm]{Proposition}
\newtheorem{lemma}[thm]{Lemma}
\newtheorem{corollary}[thm]{Corollary}
\newtheorem{definition}[thm]{Definition}
\newtheorem{remark}[thm]{Remark}

\theoremstyle{definition}
\newtheorem{appendix-proposition}{Proposition}[section]

\theoremstyle{remark}
\newtheorem{appendix-remark}{Remark}[section]

\def\itemautorefname~#1\null{%
  (#1)\null
}

\newenvironment{proof-thm}[1]{\vskip1em \noindent\textbf{Proof of \autoref#1.}}{\hfill$\square$\\}

\makeatletter
\renewcommand\subsection{\@startsection{subsection}{2}%
  \z@{.5\linespacing\@plus.7\linespacing}{.1\linespacing}%
  {\normalfont\centering\scshape}}
\makeatother

\makeatletter
\renewcommand\subsubsection{\@startsection{subsubsection}{3}%
  \z@{.5\linespacing\@plus.7\linespacing}{.1\linespacing}%
  {\normalfont\centering\scshape}}
\makeatother

\makeatletter
\def\@tocline#1#2#3#4#5#6#7{\relax
  \ifnum #1>\c@tocdepth 
  \else
    \par \addpenalty\@secpenalty\addvspace{#2}%
    \begingroup \hyphenpenalty\@M
    \@ifempty{#4}{%
      \@tempdima\csname r@tocindent\number#1\endcsname\relax
    }{%
      \@tempdima#4\relax
    }%
    \parindent\z@ \leftskip#3\relax \advance\leftskip\@tempdima\relax
    \rightskip\@pnumwidth plus4em \parfillskip-\@pnumwidth
    #5\leavevmode\hskip-\@tempdima
      \ifcase #1
       \or\or \hskip 1em \or \hskip 2em \else \hskip 3em \fi%
      #6\nobreak\relax
    \dotfill\hbox to\@pnumwidth{\@tocpagenum{#7}}\par
    \nobreak
    \endgroup
  \fi}
\makeatother

\newenvironment{proof-claim}{\textit{Proof of Claim.}}{\hfill$\square$}

\begin{document}
\title{Existence of weak solutions of the surface Beris--Edwards model}
\author{Gonzalo A.~Benavides \and Ricardo H.~Nochetto \and Mansur Shakipov}

\address{Department of Mathematics, University of Maryland College Park, MD 20742}
\email{gonzalob@umd.edu, rhn@umd.edu, shakipov@umd.edu}

\subjclass[2020]{76D03, 76A15, 35R01, 35D30, 35Q35}
\keywords{weak solution, liquid crystal, PDEs on manifolds, Beris--Edwards, Faedo--Galerkin}

\begin{abstract}
We prove the existence of weak solutions to the \emph{surface Beris--Edwards} model for nematic liquid crystals posed on a $d$-dimensional ($d \in \{2,3\}$) closed hypersurface of class $C^{2,1}$.
This \emph{thermodynamically consistent} model, recently introduced by Bouck, Nochetto and Yushutin (2024), couples the incompressible \emph{tangent Navier--Stokes} equations with a kinematic equation for the Q-tensor field that encodes the orientation of the liquid crystal particles with a general state of orientational order.
Extending ideas by Abels, Dolzmann and Liu (2014) and Guillén-González and Rodríguez-Bellido (2015) for the Beris--Edwards model in flat domains, we design a Faedo--Galerkin scheme based upon eigenfunctions of an appropriate tangent Stokes operator and tensor-valued Laplace--Beltrami operator and recover a weak solution via standard compactness arguments.
\end{abstract}

\maketitle

\tableofcontents

\section{Introduction}
Nematic Liquid Crystals (NLCs) are materials that exhibit properties from both conventional liquids and solid crystals within a certain temperature range \cite{Virga1994}.
They are characterized by having elongated rod-shaped particles with a head-to-tail symmetry, that do not possess positional order but retain orientational order.
Consequently, NLCs flow similarly to traditional fluids while enjoying optical properties typical of crystals such as light polarization, which is affected by point defects \cite{BrinkmanCladis1982}.
These are small regions in the NLC material that do not have a well-defined average orientation of the molecules; mathematically, they are represented by singularities of the associated LC director field.

NLCs fall within a broader category of materials called \emph{active matter}.
Active matter possesses microscopic constituents that consume energy in order to produce forces that overall lead to macroscopic dynamics.
In particular, nematic systems have been observed to exhibit self-similar organization properties \cite{HashemiEtAl2017}.
Point defects, in turn, can travel in time and merge or annihilate each other depending on their topological charges \cite{ToothDenniston2022,KeberEtAl2014}.
Furthermore, it was hypothesized in \cite{Nelson2002} and later observed in \cite{LagerwallScalia2012} that point defects in thin LC shells can mimic colloidal chemistry.
Due to their unique mechanical and optical properties, they are widely used in, for instance, the manufacturing of liquid crystal displays (LCDs) \cite{lc_displays}, soft robotics (liquid crystal elastomers) \cite{lce_soft_robotics_science}, and solar energy \cite{solar_energy_nature}.
Further applications of active nematic systems include \cite{hoffman2022, wang2023, vafa2022} for modeling the morphogenesis of simple organisms like the hydra, active gel systems \cite{ProstJulicherJoanny2015},
cytoskeletons \cite{SuzukiEtAl2017}, bacterial suspensions \cite{WiolandEtAl2013} and filaments \cite{CaldererGolovaty:21}.
We refer to \cite{NeedlemanDogic2017,DoostmohammadiEtAl2018} for a thorough discussion.

Due to intrinsic limitations of  NLC experiments, it is of paramount importance to obtain mathematical models that are numerically tractable but rich enough to capture complex behavior.
Simplifications to 2D models for thin (but curved) NLC shells and droplets are not satisfactory because they do not account for the non-negligible effect of curvature \cite[\textsection 2.2]{LagerwallScalia2012}.
It is therefore necessary to consider models for NLCs on surfaces.
Examples for which these models are relevant are given by domains with a high length-to-width ratio which are wide enough to freely accommodate rod-like NLC particles but thin enough to neglect variations of their orientation along the thickness.

The most popular approach within continuum mechanics for modeling materials with orientational order is the Q-tensor theory \cite{de1993physics,BorthagarayWalker2021,mottram2014introduction}.
Essentially, the idea is to associate a probability density function $\rho$ defined on the unit sphere $\SS^d$ which encodes the statistical distribution of the orientations of the LC molecules at every Eulerian point $\bx \in \RR^{d+1}$.
Due to the head-to-tail symmetry assumption on the LC molecules, the first-order moments vector of $\rho$ vanishes.
Therefore, the first non-trivial information about the LC particle distribution is given by the second-order moments matrix of $\rho$.
By subtracting an appropriate multiple of the identity matrix, one obtains a symmetric and traceless matrix $\bQ$ known as the \emph{Q-tensor} that measures the deviation of the LC from the isotropic state.
The most physically relevant quantities derived from $\bQ$ are its eigenframe and corresponding eigenvalues.

Hydrodynamic Q-tensor models of NLC on surfaces are a relatively new area of research.
To the best of our knowledge, only three articles \cite{NestlerVoigt2022,BouckNochettoYushutin2024,NitschkeVoigt2025} have considered the coupling of Q-tensor kinematics with a momentum equation (aka \emph{Beris--Edwards} model) for curved liquid crystal films.
Chronologically, the first model was proposed by Nestler and Voigt in \cite{NestlerVoigt2022} in which the Q-tensor field $\bQ$ is assumed to be \emph{conforming} (i.e.~the normal to the surface is assumed to be an eigenvector of $\bQ$);
however, the overall system lacks an energy law.
Later, Bouck, Nochetto and Yushutin derived in \cite{BouckNochettoYushutin2024} the first Beris--Edwards model for NLCs on stationary surfaces that allows for arbitrarily oriented liquid crystal particles (that is, the eigenframe of $\bQ$ is not anchored to the tangent plane of the surface).
The model is obtained by means of the \emph{Generalized Onsager's principle} \cite{Doi2011,Doi2015}, whence it is \emph{thermodynamically consistent}, and hence it obeys an energy law.
In turn, Nitschke and Voigt proposed in \cite{NitschkeVoigt2025} a more general thermodynamically consistent framework for hydrodynamics of NLCs that allows, among other features, for evolving surfaces;
their arguments are based upon the \emph{Lagrange d'Alembert's principle} of classical mechanics.
We also find it important to mention the recent work \cite{NitschkeVoigt2025-arxiv} in which, among other things, the Beris--Edwards framework \cite[\textsection 3.1]{NitschkeVoigt2025} is utilized to derive a hydrodynamic model for symmetric lipid bilayers.

The goal of this paper is to establish the existence of weak solutions for the aforementioned surface Beris--Edwards model \cite{BouckNochettoYushutin2024}.
To the best of our knowledge, our work is the first one to address the solvability of a Q-tensor hydrodynamic model for NLCs on surfaces.
Recognizing the structural similarities of the surface Beris--Edwards model \cite{BouckNochettoYushutin2024} with its flat-domain counterpart \cite{BerisEdwards1994}, we adapt ideas by Abels, Dolzmann and Liu \cite{ADL2014} and Guillén-González and Rodríguez-Bellido \cite{GG-RB2015} and design a Faedo--Galerkin method based upon eigenfunctions of appropriate tangent Stokes and tensor-valued Laplace--Beltrami operators.
We establish discrete a-priori estimates that are uniform in the degree of approximation and recover a weak solution via standard compactness arguments.

\subsection{Governing equations}\label{sec:gov-eqs}
Let $\G \subset \RR^{d+1}$ ($d=2,3$) be a closed, compact and connected hypersurface of sufficient regularity (to be specified later).
Associated to $\G$, let $\bnu = (\nu_i)_{i=1}^{d+1}: \G \rightarrow \RR^{d+1}$ denote the outward unit normal vector field to $\G$, let $\bP := \bI - \bnu \otimes \bnu$ be the projection matrix onto the tangent plane and $\bB :=  \nabla_M \bnu$ be the \emph{Weingarten} map (aka \emph{shape operator}).
\autoref{tab:diff-ops} summarizes various differential operators defined over $\G$ that will be used throughout this work.
\begin{table}
\begin{tabular}{|c|c|c|c|c|}
\hline
\multirow{2}{*}{Operator name} & \multirow{2}{*}{Notation} & \multicolumn{3}{|c|}{Definition} \\
\cline{3-5}
& & scalar-valued & vector-valued & tensor-valued \\
\hline
extrinsic gradient & $\nablaM$ & $\nablaM v := \bP \nabla v^e$ & $\nablaM \bv := \nabla \bv^e \bP$ & dyadic \\
covariant gradient & $\nablaG$ & $\nablaG v := \nablaM v$ & $\nablaG \bv := \bP \nablaM \bv$ & N/A \\
\hline
extrinsic strain rate & $\bD_M$ & N/A & $\bD_M \bv := (\nabla_M \bv + \nabla_M^T \bv)/2$ & N/A \\
extrinsic spin tensor & $\bW_M$ & N/A & $\bW_M \bv := (\nabla_M \bv - \nabla_M^T \bv)/2$ & N/A \\
covariant strain rate & $\bD_\G$ & N/A & $\bD_\G\bv := (\nablaG \bv + \nablaG^T \bv)/2$ & N/A \\
covariant spin tensor & $\bW_\G$ & N/A & $\bW_\G \bv := (\nablaG \bv - \nablaG^T \bv)/2$ & N/A \\
\hline
extrinsic divergence & $\divM$ & N/A & $\divM := \tr \nablaM$ & row-wise \\
covariant divergence & $\divG$ & N/A & $\divG := \tr \nablaG = \divM$ & row-wise \\
\hline 
Laplace--Beltrami & $\DeltaM$ & \multicolumn{3}{|c|}{$\DeltaM := \divM \nablaM$ (defined componentwise)} \\
\cline{3-5}
Surface diffusion & $\DeltaS$ & N/A & $\DeltaS := \bP \divG \bD_\G$ & N/A \\
\hline
\end{tabular}
\vspace{5mm}

\caption{\emph{Differential operators on $\G$}:
here, $(\,)^e$ denotes any sufficiently smooth extension into a tubular neighborhood around $\G$ (e.g.~\cite[eq.~(1.62)]{BonitoDemlowNochetto2020}), and N/A means either that the operator does not make sense or that it shall not be used in this work.
The extrinsic gradient $\nablaM \bA$ of a tensor-valued field $\bA$ is the third-order tensor given by $(\nabla_M \bA)_{ijk} = (\nablaM A_{ij})_k$.
}
\label{tab:diff-ops}

\end{table}

Let $\bu$ be a tangential vector field to $\G$ and $\bQ$ a Q-tensor.
\begin{itemize}
\item The \textit{quartic bulk} potential of a given Q-tensor $\bQ$ is given by
\begin{equation}\label{doublewell}
F[\bQ] := \frac{a}{2} |\bQ|^2 - \frac{b}{3} \bQ : \bQ^2 + \frac{c}{4} |\bQ|^4,
\end{equation}
where $a$, $b$ and $c$ are constants, with $c > 0$.

\item Let $\cP$ be the orthogonal projection onto the space of $(d+1)$-by-$(d+1)$ symmetric and traceless matrices, which is given by
\begin{equation}\label{eq:pointwise-proj}
\cP\bA = \frac{1}{2}(\bA + \bA^T) - \frac{\tr(\bA)}{d+1} \bI.
\end{equation}

Then,
\begin{equation}\label{proj-F'}
\cP F'[\bQ] = a \bQ - b \bQ^2 + \frac{b}{d+1} |\bQ|^2 \bI + c |\bQ|^2 \bQ,
\end{equation}
where $F'$ denotes the first variation of $F$.
\item The \emph{molecular field} associated to the Q-tensor $\bQ$ is the symmetric and traceless matrix given by
\begin{equation}\label{H-def}
\bH[\bQ] := \cP\left( L \div_M \nabla_M \bQ - F'[\bQ] \right),
\end{equation}
where $L>0$ is a material constant.

\item The \emph{star spin tensor}
\begin{equation}\label{starspintensor}
\bW_* \bu := \bB \bu \otimes \bnu - \bnu \otimes \bB \bu.
\end{equation}
\item The \emph{corotation} tensor $\bS[\bQ,\bu] := \bS_\G[\bQ,\bu] + \bS_*[\bQ,\bu]$ is given by the sum of the \emph{covariant corotation} and \textit{star corotation} tensors, correspondingly defined by
\begin{equation}\label{S-def}
\bS_\G[\bQ,\bu] := \bW_\G(\bu) \bQ - \bQ \bW_\G(\bu),
\qquad \bS_*[\bQ,\bu] := \bW_*(\bu) \bQ - \bQ \bW_*(\bu).
\end{equation}

\item The \emph{Ericksen stress} $\bSigma[\bQ]$ and the \emph{tangential Ericksen stress} $\bSigma_\G$ are skew-symmetric tensors respectively defined as follows
\begin{equation}\label{Sigma-def}
\bSigma[\bQ] := \bQ \bH[\bQ]- \bH[\bQ] \bQ,
\qquad \bSigma_\G[\bQ] := \bP \bSigma[\bQ] \bP.
\end{equation}
\item The \emph{surface Beris--Edwards} and \emph{star} forces are respectively given by
\begin{equation}
\bf_\G[\bQ] := - \bP \divG \bSigma_\G[\bQ] + \bH[\bQ] : \nabla_M \bQ, \qquad
\bf_*[\bQ] := 2 \bB \bSigma[\bQ] \bnu.
\end{equation}
\end{itemize}
The Surface Beris--Edwards model \cite{BouckNochettoYushutin2024} reads:
find a tangential vector field $\bu$, a Q-tensor $\bQ$ and a scalar pressure $\uppi$ that satisfy the following evolutionary system of surface PDEs respectively given by the \emph{kinematics of Q-tensor}, \emph{linear momentum} and \emph{incompressibility} equations:
\begin{subequations}\label{eq:BE-system}
\begin{align}
\label{LC-kinematics} \partial_t \bQ + (\nabla_M \bQ) \bu & = M \bH[\bQ] + \bS[\bQ,\bu],\\
\label{momentum-eq}\rho \left( \partial_t \bu + (\nablaG \bu) \bu + \nablaG \uppi \right) & = 2\mu \bP \divG \bD_\G\bu - (\bf_\G[\bQ] + \bf_*[\bQ]),\\
\label{incompressibility} \divG \bu & = 0,
\end{align}
\end{subequations}
where the \emph{mobility} $M$, viscosity $\mu$, and density $\rho$ are positive material constants.
These equations are supplemented with the initial conditions
\begin{equation}\label{BE-in_con}
\bu(0) = \bu_0 \quad \text{and} \quad \bQ(0) = \bQ_0.
\end{equation}
System \eqref{eq:BE-system} was derived in \cite{BouckNochettoYushutin2024} by means of the Onsager's principle \cite{Doi2011,Doi2015}, making it thermodynamically consistent.
For that reason, despite its intimidating appearance, \eqref{eq:BE-system} formally admits an energy law given by
\begin{equation}\label{eq:sbe-energy-law}
\frac{\dd}{\dd t} \Big(K[\bu] + \ELdG[\bQ]\Big) = - 2 \mu \|\bD_\G\bu\|^2 - M \|\bH\|^2,
\end{equation}
where $K[\bu]$ and $\ELdG[\bQ]$ are, respectively, the \emph{kinetic} and \emph{Landau--de Gennes} energies defined as follows:
\begin{equation}\label{eq:LandaudeGennes&kinetic}
K[\bu] := \int_\G \frac{\rho}{2} |\bu|^2, \qquad \ELdG[\bQ] := \int_\G \frac{L}{2} |\nablaM \bQ|^2 + F[\bQ].
\end{equation}
\subsection{Related models in flat domains}
The Beris--Edwards model in flat domains has been thoroughly analyzed in subdomains of the Euclidean space $\RR^d$, both with and without incorporating the so-called \emph{ratio of tumbling and alignment} parameter $\xi \in \RR$ within the model.
Existence of weak solutions and higher regularity in the whole Euclidean space have been studied for $d=2,3$ by Paicu and Zarnescu \cite{PaicuZarnescu2011,PaicuZarnescu2012} for sufficiently small $\xi$ and very recently by \cite{Zhang2026} for $d=3$ and general $\xi$.
Paicu and Zarnescu also established weak-strong uniqueness for $d=2$.

Similar results for the half-space were established by Barbera, Murata and Shibata \cite{BarberaEtAl2025}, for bounded domains---by Abels, Dolzmann and Liu \cite{ADL2014,ADL2016}, Guillén-González and Rodríguez-Bellido \cite{GG-RB2014,GG-RB2015,GG-RB2015b}, Xiao \cite{Xiao2017}, Chen and Terraneo \cite{ChenTerraneo2025} and Hieber, Hussein and Wrona \cite{HieberHusseinWrona2024}---and for the torus---by Cavaterra, Rocca, Wu and Xu \cite{CRWX16}.
In turn, more complex models obtained by considering different potentials than the quartic bulk \eqref{doublewell} have also been considered, e.g., \cite{Wilkinson2015} for a singular potential and \cite{HuangDing2015,YuningWang2018} for an anisotropic potential.

For a more detailed up-to-date discussion of the analysis of the Beris--Edwards and related models in flat domains we refer to \cite[pp.~2-3]{BarberaEtAl2025}, \cite[pp.~3-4]{HieberHusseinWrona2024} and \cite[pp.~2-3]{ChenTerraneo2025}.

\begin{remark}[dimensional scope and choice of bulk potential]\label{rem:bulk-potential-dimension}
The principal physical setting for the Surface Beris--Edwards system \eqref{eq:BE-system} is $d=2$, corresponding to a surface $\G \subset \RR^3$ and a $3\times3$ Q-tensor $\bQ$.
Our analysis also accommodates $d=3$ within the same framework.
In this case, the bulk potential \eqref{doublewell} remains admissible, although general quartic isotropic bulk potentials may contain an additional independent invariant.
For clarity of exposition, we retain \eqref{doublewell} in both dimensions and focus on the existence of weak solutions for this choice of potential.
\end{remark}

\subsection{Main results}

In this section we present our results concerning the solvability of problem \eqref{eq:BE-system}, but postpone their proofs to \autoref{sec:proof-existence} and \autoref{sec:proof-recovery_pressure}.
We begin by introducing some functional analytic concepts that are required to state our main results.

For the hypersurface $\G$ whose extrinsic embedding is sufficiently regular, we adopt the standard notation for Lebesgue spaces $\LL^p$ and \emph{Sobolev spaces} $\WW^{m,p}$ for $m \geq 0$ and $p \in [1,\infty)$, whose corresponding norms are denoted respectively by $\|\cdot\|_{0,p}$ and $\|\cdot\|_{m,p}$.
We also write $\HH^m := \WW^{m,2}$ with norm $\|\cdot\|_m := \|\cdot\|_{m,2}$; for $m=0$ we simply write $\|\cdot\|$.
Since $\G$ is compact, Sobolev embeddings (Gagliardo–Nirenberg–Sobolev, Morrey's and Rellich--Kondrachov) behave just like in flat domains \cite{HebeyRobert2008}.
For a systematic discussion of Sobolev spaces on hypersurfaces of limited regularity we refer to, e.g., \cite[Section 3]{BenavidesNochettoShakipov2025-a}.

Throughout this article, for any integrability parameter $p \in (1,\infty)$, we denote by $p^* := \frac{p}{p-1}$ its Lebesgue conjugate exponent;
notice that $p^{**} = p$.
Also, $p \in (1,2)$ if and only if $p^* \in (2,\infty)$.
For any pair of functions $v \in \LL^p$ and $w \in \LL^{p^*}$ we write $(v,w)_\G := \int_\G v \, w$.

For any scalar function space $\VV$ defined on $\G$, we denote by $\bV$ and $\bbV$ its $(d+1)$-vectorial and $(d+1)\times(d+1)$ tensorial counterparts, and use the same notation for their respective inner products and/or norms.
Finally, for any normed space $V$ its dual space is denoted by $V'$, whose induced norm is given by $\|f\|_{V'} := \sup_{0 \neq v \in V} \frac{f(v)}{\|v\|_V}$.
We denote the duality pairing between $V'$ and $V$ by $\langle \cdot, \cdot \rangle_{V' \times V}$;
when there is no ambiguity, we will simply write $\langle \cdot,\cdot \rangle$.
Let $\BB\CC_w([0,T]; V)$ denote the space of weakly-continuous functions defined on $[0,T]$ with values in $V$.

Finally, for $m \in \NN \cup \{0\}$, we introduce the following spaces
\begin{align*}
\WW^{m,p}_\# & := \{\textstyle  q \in \WW^{m,p}: \int_\G q = 0 \},\\
\bW^{m,p}_t & := \{ \bv \in \bW^{m,p}: \bv \cdot \bnu = 0 ~\text{a.e.~on $\G$} \},\\
\bW^{m,p}_{t,\sigma} & := \{\textstyle  \bv \in \bW^{m,p}_t: \int_\G \bv \cdot \nablaG q = 0, \, \forall q \in \WW^{1,p^*} \},\\
\bbW^{m,p}_S & := \{\bA \in \bbW^{m,p}: \bA = \bA^T, \tr(\bA) = 0~\text{a.e.~on $\G$}\}.
\end{align*}
Sobolev spaces on $\G$ satisfy integration-by-parts formulas that closely resemble their flat-domain counterparts.
In particular, we have the following set of results.
\begin{proposition}[{\cite[Proposition 3.9]{BenavidesNochettoShakipov2025-a}} scalar-vector integration-by-parts formula]
If $\G$ is of class $C^{1,1}$, then for each $u \in \WW^{1,1}(\G)$ and $\bvarphi \in \bC^1(\G)$ it holds that
\begin{equation}\label{eq:int-by-parts}
(u,\divG \bvarphi)_\G = - (\nablaG u, \bvarphi)_\G + (\tr(\bB) u , \bvarphi \cdot \bnu)_\G
\end{equation}  
\end{proposition}
\begin{corollary}[vector-tensor integration-by-parts formula]
If $\G$ is of class $C^{1,1}$, then for each $\bu \in \bW^{1,1}(\G)$ and $\btau \in \bbC^1(\G)$ it holds that
\begin{equation}\label{covariantIBP-matrices}
(\bu, \divG \btau)_\G = - (\nabla_M \bu , \btau)_\G + (\tr (\bB) \bu, \btau \bnu)_\G.
\end{equation}
\end{corollary}

\begin{corollary}[external integration-by-parts formula]
If $\G$ is of class $C^{1,1}$, then for each $\bsigma \in \bbW^{2,1}(\G)$ and $\btau \in \bbC^1(\G)$ it holds that
\begin{equation}\label{externalIBP-matrices}
(\div_M \nabla_M \bsigma, \btau)_\G = - (\nabla_M \bsigma , \nabla_M \btau )_\G.
\end{equation}
\end{corollary}
By testing system \eqref{eq:BE-system} with appropriate test functions and formally integrating by parts in space and time we arrive at the following definition for \emph{weak solutions} of the surface Beris--Edwards system \eqref{eq:BE-system}.
We leave the details of the calculations for \autoref{sec:formal-wf} below.
\begin{definition}[weak solution]\label{def:weak-solution}
Let $\G$ be of class $C^{2,1}$.
Then, for each $T>0$, $\bu_0 \in \bL^2_{t,\sigma}$, $\bQ_0 \in \bbH^1_S$, we say that the pair $(\bu,\bQ)$, with
\begin{align*}
\bu & \in \BB\CC_w([0,T];\bL^2_{t,\sigma}) \cap \LL^2([0,T];\bH^1_{t,\sigma}),\\
\bQ & \in \BB\CC_w([0,T];\bbH^1_S) \cap \LL^2([0,T];\bbH^2_S),
\end{align*}
is a weak solution of \eqref{eq:BE-system} in $[0,T] \times \G$ with initial conditions \eqref{BE-in_con} if:
\begin{enumerate}
\item\label{it:weak-BE-eqs}
For any $\bw \in \CC^1([0,T];\bH^2_{t,\sigma})$ and $\bPsi \in \CC^1([0,T];\bbH^1_S)$ with $\bw(T) = \mathbf{0}$ and $\bPsi(T) = \mathbf{0}$, the momentum \eqref{momentum-eq} and liquid-crystal kinematic \eqref{LC-kinematics} equations are, respectively, weakly satisfied in the following sense
\begin{multline}\label{eq:integrated-momentum}
- \rho \int_0^T (\bu,\partial_t\bw)_\G + \rho \int_0^T ( (\nablaG \bu) \bu, \bw )_\G + 2 \mu \int_0^T (\bD_\G \bu, \bD_\G \bw)_\G\\
+ \int_0^T (\bSigma_{\Gamma}[\bQ], \nablaG \bw)_\G + \int_0^T (\bH[\bQ] : \nabla_M \bQ,\bw)_\G + 2 \int_0^T (\bB \bSigma[\bQ] \bnu, \bw)_\G
= \rho(\bu_0,\bw(0))_\G,
\end{multline}
and
\begin{equation}\label{eq:integrated-LCkinematics}
- \int_0^T (\bQ, \partial_t \bPsi)_\G + \int_0^T ((\nabla_M \bQ) \bu, \bPsi)_\G - \int_0^T (\bS[\bQ,\bu],\bPsi)_\G
= M\int_0^T (\bH[\bQ],\bPsi)_\G + (\bQ_0,\bPsi(0))_\G.
\end{equation}
\item\label{it:energy-ineq}
For a.e.~$t \in (0,T)$ the following energy dissipation law corresponding to \eqref{eq:sbe-energy-law} holds:
\begin{equation}\label{eq:continuous-apriori}
K[\bu(t)] + \ELdG[\bQ(t)] + \int_0^t \int_\G 2 \mu |\bD_\G \bu|^2 + M |\bH|^2 \dd \bx \dd t' \leq K[\bu_0] + \ELdG[\bQ_0].
\end{equation}
\end{enumerate}
\end{definition}

The following two statements are the main results of our paper.

\begin{thm}[existence of weak solutions]\label{thm:existence-weak-solution}
Let $d \in \{2,3\}$ and assume $\G$ is of class $C^{2,1}$.
Then, for each $T > 0$, $\bu_0 \in \bL^2_{t,\sigma}$ and $\bQ_0 \in \bbH^1_S$, there is a weak solution $(\bu,\bQ)$ of the system \eqref{eq:BE-system} in $[0,T] \times \G$ with initial conditions \eqref{BE-in_con}.
\end{thm}

Since $\G$ has no boundary, we are able to recover the pressure field $\uppi$ with higher regularity in time than the flat bounded domain analogue of the Beris-Edwards system \cite[Remark 9]{GG-RB2015}.

\begin{thm}[recovery of the pressure field]\label{thm:pressure-recovery}
Let $d \in \{2,3\}$ and let $\G$ be of class $C^{2,1}$.
Let $(\bu,\bQ)$ be a weak solution provided by \autoref{thm:existence-weak-solution}.
Then, there exists a unique $\uppi \in \bigcap_{\varepsilon \in (0,1) } \LL^2([0,T];\LL^{(d+\varepsilon)^*}_\#)$ such that
\begin{multline}\label{eq:integrated-momentum-withpressure}
- \rho \int_0^T (\bu,\partial_t\bw)_\G + \rho \int_0^T ( (\nablaG \bu) \bu, \bw )_\G + 2 \mu \int_0^T (\bD_\G \bu, \bD_\G \bw)_\G - \rho\int_0^T (\uppi,\divG \bw)_\G\\
+ \int_0^T (\bSigma_{\Gamma}[\bQ], \nablaG \bw)_\G + \int_0^T (\bH[\bQ] : \nabla_M \bQ,\bw)_\G + 2 \int_0^T (\bB \bSigma[\bQ] \bnu, \bw)_\G
= \rho(\bu_0,\bw(0))_\G,
\end{multline}
for any $\bw \in \CC^1([0,T];\bW^{2,d+\varepsilon}_{t})$ with $\bw(T) = \mathbf{0}$.
Moreover, there exists $C = C(\G,T,\varepsilon,\rho,\mu) > 0$ such that
\begin{multline}\label{eq:pressure-apriori}
\|\uppi\|_{\LL^2([0,T];\LL^{(d+\varepsilon)^*}_\#)}
\leq C \Big( \|\bu\|_{\LL^2([0,T];\bH^1)}\|\bu\|_{\LL^\infty([0,T];\bL^2)} + \|\bu\|_{\LL^2([0,T];\bL^2)}\\ + \|\bH[\bQ]\|_{\LL^2([0,T];\bbL^2)} \|\nabla_M \bQ\|_{\LL^\infty([0,T];\bbL^2)}
+ \|\bQ\|_{\LL^\infty([0,T];\bbH^1)} \|\bH[\bQ]\|_{\LL^2([0,T];\bbL^2)} \Big).
\end{multline}
\end{thm}

The rest of this paper is structured as follows.
In \autoref{sec:spectral} we establish the existence of orthonormal bases of eigenfunctions of appropriate surface PDE operators.
Their existence is a consequence of the regularity theory for PDEs on manifolds found in \cite{BenavidesNochettoShakipov2025-b} and the classical spectral decomposition of self-adjoint compact operators \cite[\textsection D.6]{Evans2010}, \cite[Chapter 6]{Brezis2011}.
In \autoref{sec:existence}, we utilize the aforementioned bases to design a Faedo--Galerkin scheme for the approximation of the surface Beris-Edwards system \eqref{eq:BE-system}.
We prove that Galerkin solutions satisfy a discrete energy law analogous to \eqref{eq:sbe-energy-law} and a-priori estimates that are uniform in the size of the discrete system.
Via standard compactness arguments, we extract convergent subsequences in different strong and weak topologies that allow us to pass to the limit in the appropriately tested (in-time) Galerkin scheme to recover a weak solution $(\bu,\bQ)$ in the sense of \autoref{def:weak-solution} of \eqref{eq:BE-system} (cf.~\autoref{thm:existence-weak-solution}).
Because the manifold $\G$ has no boundary, we are able to recover the pressure field $\uppi$ from $(\bu,\bQ)$ via a ``postprocessing'' approach based upon \autoref{prop:Helmholtzdecomposition} (Helmholtz--Weyl decomposition on $\G$) and the well-posedness of the ultra-weak formulation of the Laplace--Beltrami operator \cite[Lemma 4.3]{BenavidesNochettoShakipov2025-a}.

\section{Preliminaries on spectral theory}\label{sec:spectral}

In this section we establish the existence of orthonormal bases of eigenfunctions that will be utilized in the Faedo--Galerkin scheme for the Beris--Edwards problem \eqref{eq:BE-system}.

\begin{proposition}[Eigendecomposition of the $-\div_M \nabla_M + I$ operator]\label{lemma:divgrad-eigen}
If $\Gamma$ is of class $C^{1,1}$, then there exist a family $\{\bAn\}_{n \in \NN} \subseteq \bbH^2_S$ that is an orthonormal basis of $\bbL^2_S$ and an orthogonal basis of $\bbH^1_S$, and a nondecreasing sequence of positive numbers with $\lim_{n \to \infty} \lambda_n = \infty$ such that
\begin{equation*}
- \div_M \nabla_M \bAn + \bAn = \lambda_n \bAn, \quad \text{on $\Gamma$}.
\end{equation*}
\begin{proof}
We first notice that the Riesz-representation theorem in $\bbH^1$ ensures that for each $\bF \in \bbL^2$, there exists a unique $\bA \in \bbH^1$ such that
\begin{equation}\label{eq:aux_divPId}
(\nabla_M \bA, \nabla_M \bR)_\G + (\bA,\bR)_\G = (\bF,\bR)_\G, \qquad \forall\, \bR \in \bbH^1.
\end{equation}
Furthermore, by choosing $\bR = v \be_i \otimes \be_j$ for fixed $i,j \in \{1,\dotsc,d+1\}$ and $v \in \HH^1$, we obtain
\begin{equation}\label{eq:aux_divPId-componentwise}
(\nabla_M A_{ij}, \nabla_M v) + (A_{ij}, v)_\G = (F_{ij}, v)_\G, \qquad \forall\, v \in \HH^1,
\end{equation}
which is a scalar $\Gamma$-Poisson problem for $A_{ij}$.
The higher regularity theory \cite[Theorem 1.4]{BenavidesNochettoShakipov2025-a} (see also \cite[Lemma 3]{BonitoDemlowNochetto2020} for $\G$ of class $C^2$) implies that $A_{ij} \in \HH^2$ and $\|A_{ij}\|_2 \leq C \|F_{ij}\|$ for each $i,j$;
hence, $\bA \in \bbH^2$ and $\|\bA\|_2 \leq C \|\bF\|$.
Furthermore, by integrating by parts (cf.~\eqref{externalIBP-matrices}), we deduce that $\bA$ strongly satisfies
\begin{equation}\label{eq:strong-aux_divPId}
T (\bA) = \bF,
\end{equation}
where $T: \bbH^2 \rightarrow \bbL^2$ is the continuous and invertible operator given by $T(\bA) := - \div_M \nabla_M \bA + \bA$.
Hence, recalling that $\bbH^2$ is compactly embedded in $\bbL^2$ \cite[Theorem 8.2]{HebeyRobert2008}, the operator $T^{-1}$ from $\bbL^2$ to $\bbL^2$ is compact, symmetric (cf.~\eqref{eq:aux_divPId}) and positive.
Since $\bbL^2$ is separable, it thus follows from \cite[\textsection D.6]{Evans2010} that there exist an orthonormal basis $\{\bAn\}_{n \in \NN}$ of $\bbL^2$ and a sequence of nonincreasing positive numbers $\{\alpha_n \}_{n \in \NN}$ with $\lim_{n \to +\infty} \alpha_n = 0$ such that $T^{-1} (\bAn) = \alpha_n \bAn$, or equivalently, $T (\bAn) = \lambda_n \bAn$ with $\lambda_n = \alpha_n^{-1}$.
By the aforementioned regularity theory \cite{BenavidesNochettoShakipov2025-a} we also know that $\bAn \in \bbH^2$.
From the identity $(\nabla_M \bAn, \nabla_M \bR)_\G + (\bAn,\bR)_\G = \int_\G T(\bAn) : \bR = \lambda_n \int_\G \bAn: \bR$, it is easy to see that $\{\bAn\}_{n \in \NN}$ is also an orthogonal basis of $\bbH^1$.

Integrating by parts (cf.~\eqref{eq:int-by-parts}) in \eqref{eq:aux_divPId-componentwise} with $F_{ij} = \lambda_n A_{ij}^{(n)}$, we deduce that each component $A_{ij}^{(n)}$ of $\bAn$ strongly satisfies
\begin{equation}\label{eq:aux-eig-componentwise}
- \div_M \nabla_M A_{ij}^{(n)} + A_{ij}^{(n)} = \lambda_n A_{ij}^{(n)}, \qquad \text{a.e.~on $\Gamma$}.
\end{equation}
Let $\wtbAn = \cP \bAn$, where $\cP$ is the pointwise projection given by \eqref{eq:pointwise-proj}.
Clearly, $\wtbAn \in \bbH^2_S$ and by linearity of \eqref{eq:aux-eig-componentwise} each of its components satisfies
\begin{equation*}
- \div_M \nabla_M \wt A_{ij}^{(n)} + \wt A_{ij}^{(n)} = \lambda_n \wt A_{ij}^{(n)}, \qquad \text{a.e.~on $\Gamma$},
\end{equation*}
which altogether is equivalent to
\begin{equation}\label{eq:projected-strong-eigenequation}
- \div_M \nabla_M \wtbAn + \wtbAn = \lambda_n \wtbAn, \quad \text{on $\Gamma$}.
\end{equation}

Notice that for each $n, m \in \NN$,
\begin{equation}\label{eq:projected-eigen-orthogonality}
\begin{aligned}
\lambda_n (\wtbAn, \wtbAm)_\G
= (T(\wtbAn), \wtbAm)_\G
& = (\nabla_M \wtbAn, \nabla_M \wtbAm)_\G + (\wtbAn,\wtbAm)_\G\\
& = (\wtbAn,T(\wtbAm))_\G
= \lambda_m (\wtbAn , \wtbAm)_\G;
\end{aligned}
\end{equation}
whence $(\wtbAn, \wtbAm)_\G = (\nabla_M \wtbAn, \nabla_M \wtbAm)_\G + (\wtbAn,\wtbAm)_\G = 0$ if $\lambda_n \neq \lambda_m$.

To account for repeated eigenvalues, we orthonormalize the projected eigenfunctions $\wtbAn$ belonging to the same eigenspace.
For each $\lambda \in \Lambda := \{\lambda_n : n \in \NN\}$, the index set $I_\lambda := \{n \in \NN : \lambda_n = \lambda\}$ is finite because $\lambda_n \to \infty$.
We apply the $\bbL^2$ Gram--Schmidt orthonormalization procedure to the finite family $S_\lambda:= \{\wtbAn : n \in I_\lambda\}$, discarding every possible zero residual, yielding a finite family $\wt S_\lambda \subseteq \bbH^2_S$ that spans the same space as $S_\lambda$.
Since $\wt S_\lambda \subseteq \vspan S_\lambda$ and $T$ is linear, each element of $\wt S_\lambda$ satisfies the eigenvalue equation \eqref{eq:projected-strong-eigenequation} with eigenvalue $\lambda$.
Moreover, the eigenvalue equation and the $\bbL^2$-orthogonality also imply that the elements of $\wt S_\lambda$ are $\bbH^1$-orthogonal.

Define $\cB := \bigcup_{\lambda \in \Lambda} \wt S_\lambda$ and fix $m \in \{0,1\}$.
For each $\bA \in \bbH^m_S$, express it in terms of the basis $\{\bAn\}_{n \in \NN}$ as $\bA = \sum_{n=1}^\infty \beta_n \bAn$ with the series converging in the $\bbH^m$-norm.
The linearity of $\cP$ implies that $\bA = \cP\bA = \sum_{n=1}^\infty \beta_n \cP\bAn \in \overline{\vspan\mathcal B}^{\,\bbH^m}$.
Hence, by relabeling the eigenpairs of $\cB$ as $\{(\bAn,\lambda_n)\}_{n \in \NN}$, with $\lambda_n \uparrow \infty$, the resulting sequence $\{\bAn\}_{n \in \NN}$ is an orthonormal basis of $\bbL^2_S$ and an orthogonal basis of $\bbH^1_S$.
This finishes the proof.

\end{proof}
\end{proposition}
We next state a Helmholtz--Weyl decomposition result for closed hypersurfaces that will be instrumental in proving the existence of the pressure field $\pi$ satisfying an appropriate weak formulation of the momentum equation \eqref{momentum-eq}.
\begin{proposition}[{\cite[Proposition 4.3]{BenavidesNochettoShakipov2025-b}} Helmholtz--Weyl decomposition on $\G$]\label{prop:Helmholtzdecomposition}
Let $p \in (1,\infty)$.
Let $m$ be a nonnegative integer and assume $\G$ is of class $C^1$ if $m=0$ or $C^{m,1}$ if $m \geq 1$.
Then $\bW^{m,p}_t$ can be decomposed as
\begin{equation}\label{H^mp1-decomposition}
\bW^{m,p}_t = \bW^{m,p}_{t,\sigma} \oplus \nablaG \WW^{m+1,p}_\#.
\end{equation}
Furthermore, this decomposition is stable in the sense that there exists a positive constant $C > 0$, depending only on $m$, $p$ and $\G$, such that for each $\bv = \bv_1 + \bv_2 \in \bW^{m,p}_t$ with $\bv_1 \in \bW^{m,p}_{t,\sigma}$ and $\bv_2 \in \nablaG \WW^{m+1,p}_\#$ it holds that
\begin{equation*}
\|\bv_1\|_{m,p} + \|\bv_2\|_{m,p} \leq C \|\bv\|_{m,p}.
\end{equation*}
For $p=2$, this decomposition is $\LL^2$-orthogonal.
\end{proposition}
The Helmholtz--Weyl decomposition gives rise to the projection onto a $\divG$-free subspace, called the Leray projection.
\begin{definition}[{\cite[Definition 4.4]{BenavidesNochettoShakipov2025-b}} Leray projection on $\G$]\label{def:Leray-projection}
Under the hypotheses of \autoref{prop:Helmholtzdecomposition}, we define the Leray projection $P_\sigma: \bL^p_t \rightarrow \bL^p_{t,\sigma}$ by $P_\sigma \bv = \bh$ for each $\bv = \bh + \bg \in \bL^p_t$, where $\bh \in \bL^p_{t,\sigma}$ and $\bg \in \nablaG \WW_\#^{1,p}$ are the unique components of $\bv$ given by \autoref{prop:Helmholtzdecomposition} (Helmholtz--Weyl decomposition on $\G$).
We also denote $P_\sigma^\perp \bv = \bg$.
\end{definition}
Notice that
\begin{equation}\label{Leray-Wmp-stability}
\|P_\sigma \bv\|_{m,p} + \|P_\sigma^\perp \bv\|_{m,p} \leq C\|\bv\|_{m,p}, \qquad \bv \in \bW^{m,p}_t
\end{equation}
and
\begin{equation*}
(\bv, \bw)_\G = (P_\sigma \bv, P_\sigma \bw)_\G + (P_\sigma^\perp \bv, P_\sigma^\perp \bw)_\G, \qquad \forall \bv \in \bL^p_t, \bw \in \bL^{p^*}_t.
\end{equation*}
Central to the Faedo--Galerkin method will be the following Stokes operator and its properties that motivated and were proven in our previous work \cite{BenavidesNochettoShakipov2025-b}.
\begin{definition}[{\cite[p.~23]{BenavidesNochettoShakipov2025-b}} Stokes operator on $\G$]\label{def:Stokes-operator}
Let $p \in (1,\infty)$ and $\G$ be of class $C^{2,1}$.
Let $A: \bW^{2,p}_{t,\sigma} \rightarrow \bL^p_{t,\sigma}$ denote the \emph{Stokes operator on $\G$} defined by $A \bv = - P_\sigma (\bP \divG \nablaG \bv)$.
\end{definition}

\begin{proposition}[{\cite[Theorem 4.11]{BenavidesNochettoShakipov2025-b}} eigenfunctions of the Stokes operator]\label{prop:Stokes-eigen}
If $\G$ is of class $C^{2,1}$, then there is a family of functions $\{\bv_n\}_{n \in \NN} \subseteq \bigcap_{2 \leq p<\infty} \bW^{2,p}_{t,\sigma}$ and positive numbers $\{\omega_n\}_{n \in \NN}$ with $\lim_{n \to \infty} \omega_n = \infty$ such that
\begin{enumerate}
\item\label{it:eigen-equation} For each $n \in \NN$, $(\bv_n,\omega_n)$ is an eigenpair of $A$:
\begin{equation*}
A \bv_n = \omega_n \bv_n, \qquad \text{a.e.~on $\G$}.
\end{equation*}
\item\label{it:L^2-basis} $\{\bv_n\}_{n \in \NN}$ is an orthonormal basis of $\bL^2_{t,\sigma}$.
\item\label{it:H^1-basis} $\{\omega_n^{-1/2}\bv_n\}_{n \in \NN}$ is an orthonormal basis of $\bH^1_{t,\sigma}$ with respect to the equivalent inner product $(\bu,\bv) \mapsto (\nablaG \bu, \nablaG \bv)_\G$.
\item\label{it:H^2-basis} $\{\omega_n^{-1} \bv_n\}_{n \in \NN}$ is an orthonormal basis of $\bH^2_{t,\sigma}$ with respect to the equivalent inner product $(\bu,\bv) \mapsto (A\bu, A\bv)_\G$, because the Poincaré inequality is valid in $\bH^1_t$ \cite[Theorem 3.3]{BenavidesNochettoShakipov2025-b}.
\end{enumerate}
\end{proposition}

\section{Existence of weak solutions}\label{sec:existence}

This section is devoted to proving \autoref{thm:existence-weak-solution} (existence of weak solutions) and \autoref{thm:pressure-recovery} (recovery of the pressure field).
In \autoref{sec:F--G} we design a Faedo--Galerkin method based upon the orthonormal bases of eigenfunctions introduced in \autoref{sec:spectral} for the approximation of the surface Beris--Edwards model \eqref{eq:BE-system}.
In \autoref{sec:energy_apriori} we establish an energy law and useful a-priori estimates satisfied by the solution of the Faedo--Galerkin scheme that are needed for the proofs of \autoref{thm:existence-weak-solution} and \autoref{thm:pressure-recovery}.
In turn, we prove our main results \autoref{thm:existence-weak-solution} and \autoref{thm:pressure-recovery} respectively in \autoref{sec:proof-existence} and \autoref{sec:proof-recovery_pressure}.

\begin{remark}[regularity on $\G$]\label{rem:reg-assumption}
From now onwards we assume that $\G$ is of class $C^{2,1}$.
\end{remark}

\subsection{Faedo--Galerkin method}\label{sec:F--G}

In this section we design a continuous-in-time Faedo--Galerkin method based upon eigenfunctions of surface analogues of the Stokes and Laplace operators to respectively approximate the velocity field $\bu$ and Q-tensor $\bQ$ of the Beris--Edwards problem \eqref{eq:BE-system}.

Let $\{\bA^{(i)}\}_{i=1}^\infty \subseteq \bbH^2_S$ and $\{\bv^{(i)}\}_{i=1}^\infty \subseteq \bH^2_{t, \sigma}$ be, respectively, the eigenfunctions of $I - \divM \nablaM$ and the Stokes operators from \autoref{lemma:divgrad-eigen} and \autoref{prop:Stokes-eigen}.
We define the finite-dimensional spaces
\begin{align*}
\cQ_n & := \vspan\{\bA^{(1)},\cdots,\bAn\} \subseteq \bbL^2_S\\
\cV_n & := \vspan\{\bv^{(1)},\cdots,\bvn\} \subseteq \bL^2_{t,\sigma}.
\end{align*}
and the associated orthogonal projectors
\begin{equation*}
\pi_n: \bbL^2_S \rightarrow \cQ_n,
\quad \text{and} \quad \sP_n:\bL^2_{t,\sigma} \rightarrow \cV_n.
\end{equation*}
Recall that these orthogonal projectors are bounded linear operators with operator norms bounded by $1$.

We define the following discrete versions of $\bSigma$ and $\bSigma_\G$ (cf.~\eqref{Sigma-def}) that will be present in the Galerkin method to be soon introduced
\begin{equation}\label{Sigmas-discrete}
\bSigma_n[\bQ] := \bQ \,\pi_n\bH[\bQ] - (\pi_n\bH[\bQ]) \,\bQ, \qquad\qquad
\bSigma_{n,\Gamma}[\bQ] := \bP \bSigma_n[\bQ] \bP;
\end{equation}
note that $\bSigma_n[\bQ]$ and $\bSigma_{n,\G}[\bQ]$ are skew-symmetric.
We seek approximations of the form
\begin{equation*}
\bun(t)(x) = \sum_{i=1}^n d_i(t) \bvi(x)
\qquad \text{and} \qquad
\bQn(t)(x) = \sum_{i=1}^n h_i(t) \bAi(x),
\end{equation*}
such that the pair $(\bun,\bQn)$ satisfies the following weak formulations for $t \in (0,T)$:
\begin{subequations}\label{galerkin}
\begin{itemize}[leftmargin=*]
\item the linear momentum equation \eqref{momentum-eq} on $\cV_n$:

\begin{multline}\label{galerkin-momentum}
\rho\Big\{ (\partial_t \bun,\bvk)_\G + ( (\nablaG \bun) \bun, \bvk )_\G \Big\} + 2 \mu \, (\bD_\G \bun, \bD_\G \bvk)_\G\\
= -(\bSigma_{n,\Gamma}[\bQn], \nablaG \bvk)_\G - (\pi_n\bH[\bQn] : \nabla_M \bQn,\bvk)_\G - 2 (\bB \bSigma_n[\bQn] \bnu, \bvk)_\G,
\end{multline}
for all $k \in \{1,\dotsc,n\}$.

\item the liquid crystal kinematic equation \eqref{LC-kinematics} on $\cQ_n$:
\begin{equation}\label{galerkin-LCkinematics}
(\partial_t \bQn, \bAl)_\G + ((\nabla_M \bQn) \bun, \bAl)_\G = M (\bH[\bQn],\bAl)_\G + (\bS[\bQn,\bun],\bAl)_\G,
\end{equation}
for all $l \in \{1,\dotsc,n\}$.
\end{itemize}
We supplement these equations with the initial conditions
\begin{equation}\label{galerkin-initialconditions}
\bun(0) = \bun_0 := \sP_n \bu_0 \quad \text{and} \quad \bQn(0) = \bQn_0 := \pi_n \bQ_0.
\end{equation}
\end{subequations}

\begin{remark}\label{rem:Galerkin-invariances}
In what follows we will use that for all $t \in [0,T)$,
\begin{enumerate}
\item\label{it:partialt-stays-inspace-Q} $\partial_t \bQn(t) \in \cQ_n$, and hence $\pi_n \partial_t \bQn = \partial_t \bQn$. 
\item\label{it:DivNabla-stays-inspace} $\div_M \nabla_M \bQn(t) = \sum_{i=1}^n h_i(t) (1-\lambda_i)\bAi$, and hence $\pi_n \div_M \nabla_M \bQn = \div_M \nabla_M \bQn$.
\end{enumerate}
\end{remark}

The system formed by collecting equations \eqref{galerkin} can be seen as a finite-dimensional system of time-dependent ODEs.
Such system, according to the Peano existence theorem \cite[p.~6]{Peano-existence-theorem-ref} has a solution on a maximal time interval $[0,T_n)$, with $T_n > 0$.
We must show that $T_n$ is independent of $n$.

\subsection{Energy law and a-priori estimates}\label{sec:energy_apriori}

In this section, we establish a discrete analogue of the energy law \eqref{eq:sbe-energy-law} and appropriate a-priori estimates that are fundamental for the proof of \autoref{thm:existence-weak-solution} (existence of weak solutions).

\begin{proposition}[discrete energy law]\label{pro:discrete-energy-law}
For each $n \in \NN$ the Galerkin scheme \eqref{galerkin} satisfies the following energy law
\begin{equation}\label{eq:discretelaw}
\frac{\dd}{\dd t} \left(E_{LdG}[\bQn(t)] + K[\bun(t)] \right) = - 2\mu \|\bD_\G \bun(t) \|^2 - M \|\pi_n \bH[\bQn(t)]\|^2 
\end{equation}
for all $t \in (0,T_n)$.
\end{proposition}
\begin{proof}
In contrast to \cite{BouckNochettoYushutin2024}, which derives \eqref{eq:sbe-energy-law} via the Generalized Onsager's principle, we now proceed by testing \eqref{galerkin} with suitable test functions and exploiting algebraic cancellations.

Below we will use the elementary identity $\frac{\dd}{\dd t} \int_\G f(t) = \int_\G \partial_t f(t)$ for $f \in C^1((0,T_n);\LL^1(\G))$, which is valid because $\G$ is stationary and $\int_\G$ defines a bounded linear functional on $\LL^1(\G)$.

We start by noticing that we can replace $\bAl$ by $\pi_n \bH[\bQn]$ in \eqref{galerkin-LCkinematics}, since the latter is a (time-dependent) linear combination of the $\{\bAl\}_{l=1}^n$.
That way, we obtain in first instance:
\begin{equation}\label{proto:DEL-1}
\begin{aligned}
(\partial_t \bQn, \pi_n \bH[\bQn])_\G + ((\nabla_M \bQn) \bun, & \pi_n \bH[\bQn])_\G\\
& = M (\bH[\bQn],\pi_n \bH[\bQn])_\G + (\bS[\bQn,\bun],\pi_n \bH[\bQn])_\G\\
& = M \|\pi_n \bH[\bQn]\|^2 + (\bS[\bQn,\bun],\pi_n \bH[\bQn])_\G.
\end{aligned}
\end{equation}
Since $\partial_t \bQn(t) \in \cQ_n$ (cf.~\autoref{it:partialt-stays-inspace-Q} of \autoref{rem:Galerkin-invariances}), the first term on the left-hand side of \eqref{proto:DEL-1} can be rewritten as a total derivative in time; more precisely
\begin{multline*}
(\partial_t \bQn, \pi_n \bH[\bQn])_\G
= (\partial_t \bQn, \bH[\bQn])_\G
\stackrel{\eqref{H-def}}{=} \left(\partial_t \bQn, \cP\left( L \div_M \nabla_M \bQn - F'[\bQn] \right) \right)_\G\\
= L (\partial_t \bQn, \div_M \nabla_M \bQn)_\G - (\partial_t \bQn, F'[\bQn])_\G
\stackrel{\eqref{externalIBP-matrices}}{=} - L (\partial_t \nabla_M \bQn, \nabla_M \bQn)_\G - (\partial_t \bQn, F'[\bQn])_\G\\
= - \frac{L}{2} \frac{\dd}{\dd t} \|\nabla_M \bQn\|^2 - \frac{\dd}{\dd t} \int_\G F[\bQn]
= - \frac{\dd}{\dd t} E_{LdG}[\bQn].
\end{multline*}
Replacing this back in \eqref{proto:DEL-1}, yields
\begin{equation}\label{proto:DEL-2}
\frac{\dd}{\dd t} E_{LdG}[\bQn] + M \|\pi_n \bH[\bQn]\|^2
= - (\bS[\bQn,\bun],\pi_n \bH[\bQn])_\G + ((\nabla_M \bQn) \bun, \pi_n \bH[\bQn])_\G.
\end{equation}
Now, since $\bun$ is a (time-dependent) linear combination of the $\bvk$, $k=1,\dotsc,n$, we can replace $\bvk$ by $\bu^{(n)}$ in \eqref{galerkin-momentum}, thereby obtaining
\begin{multline}\label{proto:DEL-3}
\frac{\dd}{\dd t} K[\bun] + 2\mu \|\bD_\G \bun\|^2
\stackrel{\eqref{eq:LandaudeGennes&kinetic}}{=} \frac{\rho}{2} \frac{\dd}{\dd t} \|\bun\|^2 + 2\mu \|\bD_\G \bun\|^2\\
= -(\bSigma_{n,\Gamma}[\bQn], \nablaG \bun)_\G - (\pi_n\bH[\bQn] : \nabla_M \bQn,\bun)_\G - 2 (\bB \bSigma_n[\bQn] \bnu, \bun)_\G,
\end{multline}
where we have used that $((\nablaG \bun) \bun, \bun )_\G = 0$, because $\bun$ is $\Gamma$-divergence-free (see, e.g.~\cite[Proposition 5.1]{BenavidesNochettoShakipov2025-b}).
In order to conclude the proof we shall rewrite each one of the three terms in the right-hand side of \eqref{proto:DEL-3}.
Regarding the first term, we have
\begin{equation}\label{proto:DEL-4}
\begin{aligned}
(\bSigma_{n,\Gamma}[\bQn], \nablaG \bun)_\G
& \stackrel{\eqref{Sigmas-discrete}}{=} (\bSigma_n[\bQn], \nablaG \bun)_\G
\stackrel{\text{\tiny \autoref{tab:diff-ops}}}{=} (\bSigma_n[\bQn], \bW_\G(\bun))_\G\\
& \stackrel{\eqref{Sigmas-discrete}}{=} (\bQn\, \pi_n\bH[\bQn] - \pi_n\bH[\bQn] \,\bQn, \bW_\G(\bun))_\G\\
& = -(\bW_\G(\bun) \bQn - \bQn \bW_\G(\bun), \pi_n\bH[\bQn])_\G\\
& \stackrel{\eqref{S-def}}{=} -(\bS_\G[\bQn,\bun], \pi_n\bH[\bQn])_\G,
\end{aligned}
\end{equation}
where we have used that $\nablaG \bun = \bP (\nablaG \bun) \bP$, the fact that $\bSigma_n[\bQn]$ is skew-symmetric and $\bQn$ and $\pi_n\bH[\bQn]$ are symmetric, and the easily verifiable identity for $\bA, \bB$ symmetric and $\bC$ skew-symmetric $(\bA \bB - \bB \bA): \bC = - \bB : (\bC \bA - \bA \bC)$.
On the other hand, upon using the identity 
\begin{equation}\label{eq:tensor-tensor-vector}
(\bR : \nabla_M \bT) \cdot \bv = \bR : (\nabla_M \bT)\bv
\end{equation}
for tensor-valued functions $\bR$ and $\bT$, and vector-valued function $\bv$ \cite[eq.~(2.6)]{BouckNochettoYushutin2024}, the second term takes the form
\begin{equation}\label{proto:DEL-5}
(\pi_n\bH[\bQn] : \nabla_M \bQn,\bun)_\G = ((\nabla_M \bQn) \bun, \pi_n\bH[\bQn])_\G.
\end{equation}
Finally, for the third term, we utilize the symmetry of the Weingarten map $\bB$, $\bQn$ and $\pi_n\bH[\bQn]$ and the skew-symmetry of $\bSigma_n[\bQn]$ and $\bW_*(\bun)$ to write
\begin{multline}\label{proto:DEL-6}
2 (\bB \bSigma_n[\bQn] \bnu, \bun)_\G
= 2 (\bSigma_n[\bQn], \bB \bun \otimes \bnu)_\G
= (\bSigma_n[\bQn], \bB \bun \otimes \bnu - \bnu \otimes  \bB \bun)_\G\\
\stackrel{\eqref{starspintensor}}{=} (\bSigma_n[\bQn], \bW_*(\bun))_\G
\stackrel{\eqref{Sigmas-discrete}}{=} (\bQn \,\pi_n\bH[\bQn] - \pi_n\bH[\bQn] \,\bQn, \bW_*(\bun))_\G\\
= - (\bW_*(\bun) \bQn - \bQn \bW_*(\bun), \pi_n\bH[\bQn])_\G
\stackrel{\eqref{S-def}}{=} - (\bS_*[\bQn,\bun],\pi_n\bH[\bQn])_\G.
\end{multline}
Upon replacing \eqref{proto:DEL-4}, \eqref{proto:DEL-5} and \eqref{proto:DEL-6} back in \eqref{proto:DEL-3}, the proof is finished by adding the resulting equality to \eqref{proto:DEL-2} and realizing that $\bS[\bQn,\bun] := \bS_\G[\bQn,\bun] + \bS_*[\bQn,\bun]$ cancels with the corresponding term in \eqref{proto:DEL-2}.
\end{proof}

We now integrate the estimate \eqref{eq:discretelaw} in time over $[0,t]$ with $t < T_n$ to obtain:
\begin{equation}\label{discrete-apriori}
E_{LdG}[\bQn(t)] + K[\bun(t)] + \int_0^t \Big(2\mu  \|\bD_\G \bun(s) \|^2 + M \|\pi_n \bH[\bQn(s)]\|^2\Big) \dd s
= E_{LdG}[\bQn_0] + K[\bun_0].
\end{equation}
Since $|\bQ|^4$ is the leading order term of $F[\bQ]$ (cf.~\eqref{doublewell}), it is easy to see that there are positive constants $c_1$, $c_2$, $c_3$ and $c_4$ depending only on $a$, $b$, $c$ and $d$ such that for every matrix $\bA$
\begin{equation}\label{eq:F-cubic-bounds}
c_1 |\bA|^4 - c_2 \leq F[\bA] \leq c_3 |\bA|^4 + c_4.
\end{equation}
Since $E_{LdG}$ and $K$ are bounded from below, we obtain from \eqref{discrete-apriori} the following discrete a priori estimates associated with the Galerkin scheme \eqref{galerkin}
\begin{multline}\label{discrete-apriori-Qtensor}
\sup_{t \in [0,T_n)} \left( \|\nabla_M \bQn(t)\|^2 + \|\bQn(t)\|_{0,4}^4 \right) + \int_0^{T_n} \|\pi_n \bH[\bQn(s)]\|^2 \dd s\\
\leq C_1 \left( \|\nabla_M \bQn_0\|^2 + \|\bQn_0\|_{0,4}^4 + \|\bun_0\|^2\right) + C_2
\end{multline}
and
\begin{equation}\label{discrete-apriori-u}
\sup_{t \in [0,T_n)} \|\bun(t)\|^2 + \int_0^{T_n} \|\bD_\G \bun(s) \|^2 \dd s
\leq C_1 \left( \|\nabla_M \bQn_0\|^2 + \|\bQn_0\|_{0,4}^4 + \|\bun_0\|^2\right) + C_2,
\end{equation}
where $C_1$ and $C_2$ are positive constants depending only on $\G$, $L$, $\rho$, $a$, $b$, $c$, $d$, $\mu$ and $M$.

\begin{thm}
Every solution $(\bun,\bQn)$ of the Galerkin scheme \eqref{galerkin} exists globally in time, namely $T_n = \infty$.
\end{thm}
\begin{proof}
From \eqref{discrete-apriori-Qtensor} and \eqref{discrete-apriori-u}, $t \to \|\bQn(t)\|_{0,4} + \|\bun(t)\|$ is uniformly bounded.
That is, the norm of the solution $(\bun,\bQn)$ (belonging to a finite-dimensional space) is uniformly bounded in time.
Then, a straightforward application of \cite[Theorem 3.1]{majda_vorticity_2002} yields $T_n = \infty$.
\end{proof}

\begin{proposition}[a-priori estimates for $(\bun,\bQn)$]\label{prop:reg-space}
Let $n \in \NN$ and let $(\bun,\bQn)$ be a solution of system \eqref{galerkin}.
Then, for all $T > 0$,
\begin{equation}\label{eq:regularity-in-space}
\begin{aligned}
\bun \in \LL^2([0,T];\bH^1_{t,\sigma}) \cap \LL^\infty([0,T]; \bL^2_{t,\sigma} )\\
\bQn \in \LL^2([0,T];\bbH^2_S) \cap \LL^\infty([0,T];\bbH^1_S)
\end{aligned}
\end{equation}
and the a priori estimate holds
\begin{equation}\label{discrete-apriori-higher}
\int_0^T \|\bun(t)\|_{1}^2 \dd t + \sup_{t \in [0,T]} \|\bun(t)\|^2
+ \int_0^T \|\bQn(t)\|_{2}^2 \dd t + \sup_{t \in [0,T]} \|\bQn(t)\|_{1}^2 \leq C,
\end{equation}
where $C>0$ depends only on $\G$, $T$, $L$, $\rho$, $a$, $b$, $c$, $d$, $\mu$, $M$ ,$\|\bu_0\|$ and $\|\bQ_0\|_{1}$.

\end{proposition}
\begin{proof}

First notice that since $\{\bAn\}_{n \in \NN}$ is an orthogonal basis of $\bbH^1_S$ (cf.~\autoref{lemma:divgrad-eigen}), the operator $\pi_n\rvert_{\bbH^1_S}$ is the orthogonal projector from $\bbH^1_S$ onto $\cQ_n$.

We shall now bound the right-hand sides of \eqref{discrete-apriori-u} and \eqref{discrete-apriori-Qtensor} independently of $n$.
Indeed, it holds that
\begin{multline}\label{RHSdata-bound-independent-of-n}
\|\nabla_M \bQn_0\|^2 + \|\bQn_0\|_{0,4}^4 + \|\bun_0\|^2
= \|\nabla_M \pi_n \bQ_0\|^2 + \|\pi_n \bQ_0\|_{0,4}^4 + \|\sP_n\bu_0\|^2\\
\leq \|\pi_n \bQ_0\|_{1}^2 + \|\imath\|^4 \|\pi_n \bQ_0\|_{1}^4 + \|\bu_0\|^2
\leq \|\bQ_0\|_{1}^2 + \|\imath\|^4 \|\bQ_0\|_{1}^4 + \|\bu_0\|^2,
\end{multline}
where $\|\imath\| $ is the boundedness constant of the continuous (and compact) Sobolev embedding $\imath: \bbH^1 \rightarrow \bbL^4$.
Now, combining the fact that $\LL^\infty([0,T];\bL^2_t)$ is continuously embedded in $\LL^2([0,T];\bL^2_t)$ (with a constant depending on $T$) and Korn's inequality in $\bH^1_t$ \cite[eq.~(4.8)]{JankuhnOlshanskiiMaximReusken2018}, yields
\begin{multline*}
\sup_{t \in [0,T]} \|\bun(t)\|^2 + \int_0^T \|\bD_\G \bun(t) \|^2 \dd t
\geq \frac{1}{2} \sup_{t \in [0,T]} \|\bun(t)\|^2 +  C_T \left( \int_0^T \|\bun(t)\|^2 + \|\bD_\G \bun(t) \|^2 \dd t \right)\\
\geq \frac{1}{2} \sup_{t \in [0,T]} \|\bun(t)\|^2 + \tilde C_T \int_0^T \|\bun(t)\|_{1}^2 \dd t.
\end{multline*}
This, combined with \eqref{discrete-apriori-u} and \eqref{RHSdata-bound-independent-of-n}, proves the desired bound on $\bun$.
To prove the desired estimate for $\bQn$ we need to improve the bound on $\pi_n \bH[\bQn]$ (cf.~\eqref{discrete-apriori-Qtensor}) to a uniform bound on $\div_M \nabla_M \bQn$.
Recall the definition \eqref{H-def} of $\bH[\bQn]$, which since $ \div_M \nabla_M \bQn \in \cQ_n$  (cf.~\autoref{it:DivNabla-stays-inspace} of \autoref{rem:Galerkin-invariances}), lets us write
\begin{equation}\label{pin-of-H-restrictedtoGalerkin}
\pi_n \bH[\bQn] = L  \div_M \nabla_M \bQn - \pi_n \cP F'[\bQn].
\end{equation}
Note that combining \eqref{discrete-apriori-Qtensor} and \eqref{RHSdata-bound-independent-of-n} leads to
\begin{equation}\label{auxiliar-L^infty_H1-boundforQn}
\sup_{t \in [0,T]} \|\bQn(t)\|_{1} \leq C.
\end{equation}
In turn, combining identity \eqref{pin-of-H-restrictedtoGalerkin} and estimates \eqref{discrete-apriori-Qtensor} and \eqref{RHSdata-bound-independent-of-n} implies
\begin{equation*}
\int_0^T \|\div_M \nabla_M \bQn(t)\|^2 \dd t \leq C\left(1 + \int_0^T \|F'[\bQn](t) \|^2 \dd t\right),
\end{equation*}
for some $C>0$ depending on $\|\bQ_0\|_1$ and $\|\bu_0\|$.
Moreover, since $F'[\bQn]$ has at most cubic terms in $\bQn$, the Sobolev embedding $\bbH^1 \hookrightarrow \bbL^6$ (because $d=2$ or $d=3$) in conjunction with \eqref{auxiliar-L^infty_H1-boundforQn} further yields
\begin{equation*}
\int_0^T \| \div_M \nabla_M \bQn(t)\|^2 \dd t \leq C_T.
\end{equation*}
We finally notice that by elliptic regularity of the operator $\div_M \nabla_M - I$ (see the proof of \autoref{lemma:divgrad-eigen}) it holds that
\begin{equation*}
\|\bQn(t)\|_{2}^2 \leq C\left( \| \div_M \nabla_M \bQn(t)\|^2 + \|\bQn(t)\|^2 \right), \qquad \text{a.e.~} t \in [0,T].
\end{equation*}
This concludes the proof.
\end{proof}

Before establishing further a-priori estimates for $(\bun,\bQn)$ we require some interpolation inequalities, which due to the lack of boundary of $\G$, are akin to their flat domain counterparts (the so-called Ladyzhenskaya's inequality).

\begin{lemma}[Ladyzhenskaya's inequalities]
If $d=2$, then
\begin{equation}\label{lady-2D}
\|v\|_{0,4} \lesssim \|v\|_{1}^{1/2} \|v\|^{1/2}, \qquad \forall v \in \HH^1(\G),
\end{equation}
whereas if $d=3$, then
\begin{equation}\label{lady-3D}
\|v\|_{0,3} \lesssim \|v\|_{1}^{1/2} \|v\|^{1/2}, \qquad \forall v \in \HH^1(\G),
\end{equation}
and
\begin{equation}\label{lady-3D-L4}
\|v\|_{0,4} \lesssim \|v\|_{1}^{3/4} \|v\|^{1/4}, \qquad \forall v \in \HH^1(\G).
\end{equation}

\begin{proof}
Since $\G$ is compact there is a finite atlas $\{(\cV_i,\cU_i,\bchi_i)\}_{i=1}^N$, where each of the charts $\bchi_i: \cV_i \rightarrow \cU_i \cap \G$ is an isomorphism of the same regularity as $\G$ and compatible with its orientation.
Without loss of generality we can assume that there exist domains $\cW_i \subseteq \RR^{d+1}$ such that $\overline{\cW_i} \subseteq \cU_i$ and $\{\cW_i\}_{i=1}^N$ is still a covering of $\G$.
Furthermore, by considering a partition of unity $\{\theta_i\}_{i=1}^N$ associated with the covering $\{\cW_i\}_{i=1}^N$ of $\G$ it follows that for each $i=1,\dotsc,N$ the function $v_i := \theta_i v \in \HH^1(\G)$ is such that $v_i \circ \bchi_i \in \HH^1_0(\cV_i)$.
Let us prove the case $d=2$;
the case $d=3$ will follow similarly.
By the triangle's inequality, Ladyzhenskaya's inequality in flat domains \cite{Nirenberg1959} and Cauchy-Schwarz inequality we have that
\begin{equation*}
\|v\|_{0,4}
\leq \sum_{i=1}^N \|v_i\|_{0,4}
\cong \sum_{i=1}^N \|v_i \circ \bchi_i \|_{\LL^4(\cV_i)}
\lesssim \sum_{i=1}^N \|v_i \circ \bchi_i \|_{\HH^1(\cV_i)}^{1/2} \|v_i \circ \bchi_i \|_{\LL^2(\cV_i)}^{1/2}
\cong \sum_{i=1}^N \|v_i\|_{1}^{1/2} \|v_i\|^{1/2}.
\end{equation*}
Finally, by noticing that $\|v_i\|_{1} \lesssim \|v\|_{1}$ and $\|v_i\| \lesssim \|v\|$ because $\theta_i \in C^1(\G)$, we obtain that
\begin{equation*}
\|v\|_{0,4} \lesssim \sum_{i=1}^N \|v\|_{1}^{1/2} \|v\|^{1/2} =  N \|v\|_{1}^{1/2} \|v\|^{1/2},
\end{equation*}
which finishes the proof.
\end{proof}
\end{lemma}

\begin{proposition}[a-priori estimates in time for $(\partial_t\bun,\partial_t\bQn)$]\label{prop:reg-time}
Let $n \in \NN$ and let $(\bun,\bQn)$ be a solution of system \eqref{galerkin}.
Then, for all $T > 0$ we have the following estimate:
\begin{equation*}
\int_0^T \|\partial_t\bun\|_{(\bH^2_{t,\sigma})'}^2 \dd t
+ \int_0^T \|\partial_t \bQn\|_{(\bbH^1_S)'}^2 \dd t
\leq C,
\end{equation*}
where $C > 0$ is a constant depending only on $T$, $L$, $\rho$, $a$, $b$, $c$, $d$, $\mu$, $M$, $\G$, $\|\bu_0\|$ and $\|\bQ_0\|_{1}$.
\end{proposition}

\begin{proof}
We start by proving the estimate for $\partial_t \bun$.
By H\"older's inequality and the Sobolev embeddings $\bH^2 \hookrightarrow \bL^\infty$, $\bbH^1 \hookrightarrow \bbL^6$ and $\bbH^1 \hookrightarrow \bbL^3$, we have that for each $\bv \in \cV_n$ and for each $t \in (0,T]$,
\begin{equation*}
|(\bD_\G \bun, \bD_\G \bv)_\G|
\leq \|\nablaG\bun\| \|\nablaG \bv\|
\leq \|\bun\|_{1} \|\bv\|_{2},
\end{equation*}
\begin{equation*}
|(\bSigma_{n,\Gamma}[\bQn], \nablaG \bv)_\G|
\stackrel{\eqref{Sigmas-discrete}}{\leq} 2 \|\bQn\|_{0,6} \|\pi_n \bH[\bQn]\| \|\nablaG \bv\|_{0,3}
\lesssim \|\bQn\|_{1} \|\pi_n \bH[\bQn]\| \|\bv\|_{2},
\end{equation*}
\begin{equation*}
|(\pi_n\bH[\bQn] : \nabla_M \bQn,\bv)_\G|
\leq \|\pi_n\bH[\bQn]\| \|\nabla_M \bQn\| \|\bv\|_{0,\infty}
\lesssim \|\pi_n \bH[\bQn]\| \|\bQn\|_{1} \|\bv\|_{2},
\end{equation*}
and
\begin{equation*}
|(\bB \bSigma_n[\bQn] \bnu, \bv)_\G|
\leq \|\bB\|_{0,\infty} \|\bQn \,\pi_n\bH[\bQn] - \pi_n\bH[\bQn] \,\bQn\|_{0,1} \|\bv\|_{0,\infty}
\lesssim \|\bB\|_{0,\infty} \|\bQn\| \|\pi_n \bH[\bQn]\| \|\bv\|_{2}.
\end{equation*}
Depending on the dimension $d$, we control the convective term differently.
For $d = 2$, using H\"older's inequality and interpolation inequality \eqref{lady-2D} we arrive at
\begin{subequations}\label{bounding-convective-term}
\begin{equation}\label{bounding-convective-term-2D}
|( (\nablaG \bun) \bun, \bv )_\G|
= |( (\nablaG \bv) \bun, \bun )_\G|
\leq \|\bun\|_{0,4}^2 \|\nablaG \bv\|
\lesssim \|\bun\| \|\bun\|_{1}  \|\bv\|_{1}.
\end{equation}
If $d = 3$ we proceed similarly by using instead interpolation inequality \eqref{lady-3D} and the embedding $\bbH^1 \hookrightarrow \bbL^3$, thus arriving at
\begin{equation}\label{bounding-convective-term-3D}
|( (\nablaG \bun) \bun, \bv )_\G|
= |( (\nablaG \bv) \bun, \bun )_\G|
\leq \|\bun\|_{0,3}^2 \|\nablaG \bv\|_{0,3}
\lesssim \|\bun\| \|\bun\|_{1}  \|\bv\|_{2}.
\end{equation}
\end{subequations}
Hence, upon recalling the form of the discrete momentum equation \eqref{galerkin-momentum}, we deduce that for each $t \in (0,T]$:
\begin{equation*}
|(\partial_t \bun,\bv)_\G| \leq C \alpha_n \|\bv\|_{2}, \qquad \forall \bv \in \cV_n.
\end{equation*}
where $C$ is a positive constant depending on $\Gamma$, $\|\bB\|_{0,\infty}$, $\rho$ and $\mu$, and $\alpha_n$ is a time-dependent function defined by
\begin{equation*}
\alpha_n := (1+\|\bun\|) \|\bun\|_{1} + \|\pi_n \bH[\bQn]\| \|\bQn\|_{1},
\end{equation*}
which is in $\LL^2([0,T])$ with an $\LL^2$-norm bounded independently of $n$ (cf.~\eqref{discrete-apriori-Qtensor}, \eqref{discrete-apriori-higher}).
Moreover, keeping in mind that $(\partial_t \bun, \bv)_\G = 0$ for all $\bv$ in the $\bL^2$-orthogonal complement of $\cV_n$ (because $\bun(t) \in \cV_n$ for each $t$) and that $\|\sP_n \bv\|_{2} \lesssim \|\bv\|_{2}$ for each $\bv \in \bH^2_{t,\sigma}$ (because of \autoref{prop:Stokes-eigen}), we obtain that for each $t \in (0,T]$,
\begin{equation}\label{eq:bounding-partial_t-u}
\|\partial_t \bun\|_{(\bH^2_{t,\sigma})'}
= \sup_{\substack{\bv \in \bH^2_{t,\sigma}\\\bv \neq \mathbf{0}}} \frac{|(\partial_t \bun,\bv)_\G|}{\|\bv\|_{2}}
= \sup_{\substack{\bv \in \bH^2_{t,\sigma}\\\sP_n \bv \neq \mathbf{0}}} \frac{|(\partial_t \bun,\sP_n \bv)_\G|}{\|\bv\|_{2}}
\lesssim \sup_{\substack{\bv \in \bH^2_{t,\sigma}\\\sP_n \bv \neq \mathbf{0}}} \frac{|(\partial_t \bun,\sP_n \bv)_\G|}{\|\sP_n \bv\|_{2}}
\lesssim \alpha_n.
\end{equation}
which proves the estimate for $\partial_t \bun$.

Now we prove the estimate for $\partial_t \bQn$.
By H\"older's inequality and Sobolev embeddings $\bbH^1 \hookrightarrow \bbL^6$ and $\bH^1 \hookrightarrow \bL^6$ we have that for each $t \in (0,T]$ and $\bA \in \cQ_n$,
\begin{equation*}
|((\nabla_M \bQn) \bun, \bA)_\G|
\leq \|\nabla_M \bQn\| \|\bun\|_{0,3} \|\bA\|_{0,6}
\lesssim \|\bQn\|_{1} \|\bun\|_{1} \|\bA\|_{1},
\end{equation*}
\begin{equation*}
|(\pi_n\bH[\bQn],\bA)_\G|
\leq \|\pi_n \bH[\bQn]\| \|\bA\|_{1}
\end{equation*}
and
\begin{multline*}
|(\bS[\bQn,\bun],\bA)_\G|
\stackrel{\eqref{S-def}}{\leq} 2(\|\bW_\G[\bun]\| + \|\bW_*[\bun]\|) \|\bQn\|_{0,3} \|\bA\|_{0,6}
\lesssim (\|\bW_\G[\bun]\| + \|\bW_*[\bun]\|) \|\bQn\|_{1} \|\bA\|_{1}\\
\stackrel{\text{\tiny \autoref{tab:diff-ops}},\eqref{starspintensor}}{\leq} C(\|\nablaG \bun\| + \|\bB\|_{0,\infty} \|\bun\| ) \|\bQn\|_{1} \|\bA\|_{1}
\leq C(1 + \|\bB\|_{0,\infty}) \|\bQn\|_{1} \|\bun\|_{1} \|\bA\|_{1}.
\end{multline*}
All these estimates combined with the discrete kinematic equation \eqref{galerkin-LCkinematics} yield,
\begin{equation*}
|(\partial_t \bQn,\bA)_\G|
\leq C b_n \|\bA\|_{1}, \qquad \forall t \in (0,T],
\end{equation*}
where $C>0$ only depends on $M$ and $\G$, and $b_n$ is the time-dependent function defined by
\begin{equation*}
b_n := \|\bQn\|_{1} \|\bun\|_{1} + \|\pi_n \bH[\bQn]\|,
\end{equation*}
which is in $\LL^2([0,T])$ with an $\LL^2$-norm bounded independently of $n$ (cf.~\eqref{discrete-apriori-Qtensor},\eqref{discrete-apriori-higher}).
Since $(\partial_t \bQn, \bA)_\G = 0$ for all $\bA$ in the $\bbL^2$-orthogonal complement of $\cQ_n$ (because $\bQn(t) \in \cQ_n$ for each $t$) and that $\|\pi_n \bA\|_{1} \leq C \|\bA\|_{1}$ for each $\bA \in \bbH^1_S$ (because of \autoref{lemma:divgrad-eigen}), we proceed similarly to \eqref{eq:bounding-partial_t-u} to obtain that for each $t \in (0,T]$,
\begin{equation*}
\|\partial_t \bQn\|_{(\bbH^1_S)'} \lesssim b_n,
\end{equation*}
which proves the estimate for $\partial_t \bQn$.
\end{proof}

\subsection{Proof of \autoref{thm:existence-weak-solution}: existence of weak solutions}\label{sec:proof-existence}
We proceed similarly to \cite[Proof of Theorem 1.2]{ADL2014} (see also \cite{GG-RB2015}) by standard compactness arguments.
In \autoref{sec:1step-subseq} we utilize the classical Aubin--Lions--Simon and Banach--Alaoglu theorems and interpolation inequalities to extract subsequences of the sequence of Galerkin solutions that converge in appropriate weak and strong topologies to a common limit $(\bu,\bQ)$.
With these convergences at hand, we are able to pass to the limit in the suitably tested-in-time equations \eqref{galerkin} satisfied by the Galerkin solutions in order to respectively prove in \autoref{sec:2step-momentum}, \autoref{sec:3step-LCkinematics} and \autoref{sec:4step-Eineq} that the limit $(\bu,\bQ)$ satisfies the \emph{weak momentum equation} \eqref{eq:integrated-momentum}, \emph{weak LC kinematics equation} \eqref{eq:integrated-LCkinematics} and energy inequality \eqref{eq:continuous-apriori};
hence, $(\bu,\bQ)$ is a weak solution of the surface Beris--Edwards problem.

\subsubsection{First step: convergence up to subsequence}\label{sec:1step-subseq}

Let $T > 0$.
From \autoref{prop:reg-space} and \autoref{prop:reg-time} we know that the sequence of discrete solutions $(\bun,\bQn)$ satisfies the bound
\begin{equation}\label{galerkin-bound-altogether}
\begin{aligned}
\int_0^T \|\bun(t)\|_{1}^2 \dd t & + \sup_{t \in [0,T]} \|\bun(t)\|^2
+ \int_0^T \|\partial_t\bun\|_{(\bH^2_{t,\sigma})'}^2 \dd t\\
& + \int_0^T \|\bQn(t)\|_{2}^2 \dd t + \sup_{t \in [0,T]} \|\bQn(t)\|_{1}^2
+ \int_0^T \|\partial_t \bQn\|_{(\bbH^1_S)'}^2 \dd t
\leq C,
\end{aligned}
\end{equation}
where $C>0$ depends on data and $T$, but is independent of $n$.
By the Banach--Alaoglu Theorem \cite[Theorems 3.16 and 3.18]{Brezis2011} and the reflexivity of the corresponding Bochner spaces for $p \in (1,\infty)$, we have that (up to a subsequence)
\begin{subequations}\label{eq:weak-conv-Q&u}
\begin{align}
\label{eq:weak-conv-u1} \bun & \rightharpoonup \bu,  & \text{in $\LL^2([0,T];\bH^1_{t,\sigma})$},\\
\label{eq:weak-conv-u2} \bD_\G \bun & \rightharpoonup \bD_\G \bu,  & \text{in $\LL^2([0,T];\bbL^2)$},\\
\label{eq:weak-conv-u3} \bun & \xrightharpoonup{*} \bu,  & \text{in $\LL^\infty([0,T];\bL^2_{t,\sigma})$},\\
\label{eq:weak-conv-Q1}\bQn & \rightharpoonup \bQ,  & \text{in $\LL^2([0,T];\bbH^2_S)$},\\
\label{eq:weak-conv-Q2} \nabla_M \bQn & \rightharpoonup \nabla_M \bQ,  & \text{in $\LL^2([0,T];\HH^1(\G, \RR^{(d+1)^3}))$},\\
\label{eq:weak-conv-Q3} \div_M \nabla_M \bQn & \rightharpoonup \div_M \nabla_M \bQ, & \text{in $\LL^2([0,T];\bbL^2_S)$},\\
\label{eq:weak-conv-Q4} \bQn & \xrightharpoonup{*} \bQ,  & \text{in $\LL^\infty([0,T];\bbH^1_S)$},\\
\label{eq:weak-conv-Q5} \nabla_M \bQn & \xrightharpoonup{*} \nabla_M \bQ,  & \text{$\LL^\infty([0,T];\LL^2(\G, \RR^{(d+1)^3}))$};
\end{align}
\end{subequations}
notice that we have used in \eqref{eq:weak-conv-Q3} that $\div_M \nabla_M \bQ \in \bbL^2_S$ because $\div_M \nabla_M \bQn \in \cQ_n \subseteq \bbL^2_S$ for all $n$.
Moreover, taking into account the sequence of Sobolev embeddings $\bbH^2_S \xhookrightarrow{c} \bbH^1_S \hookrightarrow (\bbH^1_S)'$
and $\bbH^1_S \xhookrightarrow{c} \bbL^p_S \hookrightarrow (\bbH^2_S)'$ for $p \in (1,6)$, a direct application of Aubin--Lions--Simon theorem \cite[p.~85, Corollary 4]{Simon1987} to estimates \eqref{galerkin-bound-altogether} yields the following strong convergences (up to a subsequence) for the sequence of Q-tensors $\{\bQn\}_{n \in \NN}$:
\begin{subequations}\label{eq:strong-conv-u&Q}
\begin{align}
\label{eq:strong-conv-Q1} \bQn & \to \bQ, \qquad \text{in $\LL^2([0,T];\bbH^1_S)$},\\
\label{eq:strong-conv-Q2} \bQn & \to \bQ, \qquad \text{in $\CC([0,T];\bbL^p_S)$}. %
\end{align}
Similarly, given the sequence of Sobolev embeddings $\bH^1_{t,\sigma} \xhookrightarrow{c} \bL^2_{t,\sigma} \hookrightarrow (\bH^2_{t,\sigma})'$ and $\bL^2_{t,\sigma} \xhookrightarrow{c} (\bH^1_{t,\sigma})' \hookrightarrow (\bH^2_{t,\sigma})'$, we also obtain the following strong convergences for $\{\bun\}_{n \in \NN}$:
\begin{align}
\label{eq:strong-conv-u1} \bun & \to \bu, \qquad \text{in $\LL^2([0,T];\bL^2_{t,\sigma})$},\\
\label{eq:strong-conv-u2} \bun & \to \bu, \qquad \text{in $\CC([0,T];(\bH^1_{t,\sigma})')$}.
\end{align}
\end{subequations}
Since $\bu \in \LL^\infty([0,T];\bL^2_{t,\sigma}) \cap \CC([0,T];(\bH^1_{t,\sigma})')$ and the embedding $\bL^2_{t,\sigma} \hookrightarrow (\bH^1_{t,\sigma})'$ is continuous, we deduce from \cite[p.~178, Lemma 1.4]{Temam2001} that $\bu \in \BB\CC_w([0,T];\bL^2_{t,\sigma})$.
Similarly, since $\bQ \in \LL^\infty([0,T];\bbH^1_S) \cap \CC([0,T];\bbL^2_S)$ we deduce that $\bQ \in \BB\CC_w([0,T];\bbH^1_S)$.
We also claim that
\begin{equation}\label{H-weak-conv}
\pi_n \bH[\bQn] \rightharpoonup \bH[\bQ], \qquad \text{in $\LL^2([0,T];\bbL^2_S)$}.
\end{equation}
In fact, recall from definition \eqref{H-def} and identity \eqref{pin-of-H-restrictedtoGalerkin} that
\begin{align*}
\pi_n \bH[\bQn] & = L  \div_M \nabla_M \bQn - \pi_n \cP F'[\bQn],\\
\bH[\bQ] & = \cP\left( L \div_M \nabla_M \bQ - F'[\bQ] \right) = L \div_M \nabla_M \bQ - \cP F'[\bQ].
\end{align*}
Hence, given the weak convergence \eqref{eq:weak-conv-Q3}, in order to prove \eqref{H-weak-conv} it suffices to show that $\pi_n \cP F'[\bQn] \rightharpoonup \cP F'[\bQ]$ in $\LL^2([0,T];\bbL^2_S)$.
Indeed, since $\{\bQn\}_{n \in \NN}$ is bounded in $\LL^\infty([0,T];\bbH^1)$ and $\cP F'[\bQn]$ is a cubic polynomial on $\bQn$ in view of \eqref{proj-F'}, the continuous injection $\bbH^1 \hookrightarrow \bbL^6$ implies that $\{\cP F'[\bQn]\}_{n \in \NN}$ is bounded in $\LL^2([0,T];\bbL^2_S)$ and so in $\bbL^2_S([0,T]\times \Gamma)$ as well.
Moreover, notice that because of \eqref{eq:strong-conv-Q2}, it follows that (up to a subsequence) $\cP F'[\bQn] \to \cP F'[\bQ]$ a.e.~in $[0,T] \times \G$.
Invoking \cite[Section 8.2, Theorem 12]{RoydenFitzpatrick2010}, these last two properties imply that $\cP F'[\bQn]$ converges weakly to $\cP F'[\bQ]$ in $\bbL^2_S([0,T] \times \G)$ and so it does weakly in $\LL^2([0,T];\bbL^2_S)$.
Finally, due to the pointwise convergence of the self-adjoint operator $\pi_n$ to the identity map on $\bbL^2_S$ and using \cite[Proposition 3.5, (iii)]{Brezis2011} it can be proved that $\pi_n \cP F'[\bQn]$ also weakly converges to $\cP F'[\bQ]$ in $\LL^2([0,T];\bbL^2_S)$;
hence we obtain \eqref{H-weak-conv}.

Finally, regarding the nonlinear terms, we claim that $\{\bun \otimes \bun\}_{n \in \NN}$ is bounded in
\begin{equation*}
\begin{cases}
\LL^{2}([0,T];\bbL^2), & \text{if $d = 2$},\\
\LL^{\frac{4}{3}}([0,T];\bbL^2), & \text{if $d = 3$}.
\end{cases}
\end{equation*}
In fact, for $d=2$, combining interpolation inequality \eqref{lady-2D} and H\"older's inequality in time yields
\begin{equation*}
\left[\int\limits_0^T \|\bun \otimes \bun\|^2\right]^{\frac{1}{2}}
= \left[\int\limits_0^T \|\bun\|_{0,4}^4 \right]^{\frac{1}{2}}
\lesssim \left[\int\limits_0^T \|\bun\|_{1}^2 \|\bun\|^2 \right]^{\frac{1}{2}}
\leq \|\bun\|_{\LL^2([0,T];\bH^1)} \|\bun\|_{\LL^\infty([0,T];\bL^2)}
\stackrel{\eqref{galerkin-bound-altogether}}{\leq} C.
\end{equation*}
Analogously, for $d=3$, using interpolation inequality \eqref{lady-3D-L4},
\begin{equation*}
\left[\int\limits_0^T \|\bun \otimes \bun\|^{\frac{4}{3}}\right]^{\frac{3}{4}}
= \left[\int\limits_0^T \|\bun\|_{0,4}^{\frac{8}{3}} \right]^{\frac{3}{4}}
\lesssim \left[\int\limits_0^T \|\bun\|_{1}^2 \|\bun\|^{\frac{2}{3}} \right]^{\frac{3}{4}}
\leq \|\bun\|_{\LL^2([0,T];\bH^1)}^{\frac{3}{2}} \|\bun\|_{\LL^\infty ([0,T];\bL^2)}^{\frac{1}{2}}
\stackrel{\eqref{galerkin-bound-altogether}}{\leq} C,
\end{equation*}
whence (up to a subsequence)
\begin{equation}\label{eq:strong-conv-u3}
\bun \otimes \bun \rightharpoonup \bu \otimes \bu, \qquad \text{in $\begin{cases}
\LL^{2}([0,T];\bbL^2), & \text{if $d = 2$},\\
\LL^{\frac{4}{3}}([0,T];\bbL^2), & \text{if $d = 3$};
\end{cases}$}
\end{equation}
the weak limit coinciding with $\bu \otimes \bu$ because $\bun \otimes \bun \to \bu \otimes \bu$ a.e.~in $[0,T] \times \G$ (cf.~\eqref{eq:strong-conv-u1}) and by \cite[Section 8.2, Theorem 12]{RoydenFitzpatrick2010}.

\subsubsection{Step 2: momentum equation satisfied by \texorpdfstring{$(\bu,\bQ)$}{(u,Q)}}\label{sec:2step-momentum}
Let $\bw \in C^1([0,T];\bH^2_{t,\sigma})$ be such that $\bw(T) = 0$.
For any $t \in [0,T]$, we write
\begin{equation*}
\bw(t) = \sum_{i=1}^\infty d_i(t) \bvi,
\end{equation*}
with $d_i(t) = (\bw(t),\bvi)_\G$ and the series converging in the $C([0,T];\bH^m)$-norms, $m \in \{0,1,2\}$ (cf.~\autoref{prop:Stokes-eigen}).
Notice that $d_i \in C^1([0,T])$ and $d_i(T) = 0$ for each $i \in \NN$.
Moreover, it holds that $\partial_t \bw(t) = \sum_{i=1}^\infty d_i'(t) \bvi \in C([0,T];\bH^2_{t,\sigma})$ with the series also converging in the previously mentioned norms (this follows from combining the Banach--Steinhaus theorem \cite[Theorem 2.2]{Brezis2011} and the fact that $\{\partial_t \bw(t): t \in [0,T]\}$ is a compact subset of $\bH^2_{t,\sigma}$).

For a fixed $N \leq n$ define $\bwN := \sum_{i=1}^N d_i \bvi$.
Taking $\bwN(t) \in \cV_n$ for each $t \in [0,T]$ in \eqref{galerkin-momentum}, integrating in time (including an integration-by-parts in the first term), and using the initial condition $\bun(0) = \bun_0$, we arrive at
\begin{multline}\label{integrated-galerkin-momentum}
- \rho \int\limits_0^T (\bun,\partial_t\bwN)_\G + \rho \int\limits_0^T ( (\nablaG \bun) \bun, \bwN )_\G + 2 \mu \int\limits_0^T (\bD_\G \bun, \bD_\G \bwN)_\G
+ \int\limits_0^T (\bSigma_{n,\Gamma}[\bQn], \nablaG \bwN)_\G\\ + \int\limits_0^T (\pi_n\bH[\bQn] : \nabla_M \bQn,\bwN)_\G + 2 \int\limits_0^T (\bB \bSigma_n[\bQn] \bnu, \bwN)_\G
= \rho(\bun_0,\bwN(0))_\G.
\end{multline}
By utilizing convergences \eqref{eq:strong-conv-u1} and \eqref{eq:weak-conv-u2}, we can first pass to the limit as $n \to \infty$ and then as $N \to \infty$ in the first and third terms of \eqref{integrated-galerkin-momentum}, thus yielding
\begin{subequations}\label{momentum-to-limit-ALL}
\begin{equation}\label{momentum-to-limit-1}
\begin{aligned}
\int_0^T (\bun,\partial_t\bwN)_\G & \xrightarrow{n\to\infty} \int_0^T (\bu,\partial_t\bwN)_\G \xrightarrow{N \to \infty} \int_0^T (\bu,\partial_t\bw)_\G,\\
\int_0^T (\bD_\G \bun, \bD_\G \bwN)_\G & \xrightarrow{n\to\infty} \int_0^T (\bD_\G \bu, \bD_\G \bwN)_\G \xrightarrow{N\to\infty} \int_0^T (\bD_\G \bu, \bD_\G \bw)_\G.
\end{aligned}
\end{equation}
Regarding the second term in \eqref{integrated-galerkin-momentum}, using the in-space integration-by-parts formula \cite[Lemma 3.3]{BouckNochettoYushutin2024} (valid by density) and \eqref{eq:strong-conv-u3} we obtain
\begin{multline}\label{momentum-to-limit-2}
\int_0^T ( (\nablaG \bun) \bun, \bwN )_\G
= - \int_0^T ( \nablaG \bwN, \bun \otimes \bun )_\G
\xrightarrow{n\to\infty}
- \int_0^T ( \nablaG \bwN, \bu \otimes \bu)_\G\\
\xrightarrow{N \to \infty}
- \int_0^T ( \nablaG \bw, \bu \otimes \bu)_\G
= \int_0^T ( (\nablaG \bu) \bu, \bw )_\G.
\end{multline}
On the other hand, recalling definition \eqref{Sigmas-discrete} of $\bSigma_{n,\G}$ and using convergences $\bQn \to \bQ$ in $\bbL^4([0,T] \times \Gamma)$ (cf.~\eqref{eq:strong-conv-Q2}) and \eqref{H-weak-conv}, we apply \cite[Proposition 3.5, (iv)]{Brezis2011} to get
\begin{multline}\label{momentum-to-limit-3}
\int_0^T (\bSigma_{n,\Gamma}[\bQn], \nablaG(\bwN))_\G
= \int_0^T (\bSigma_{n,\Gamma}[\bQn], \bW_\G(\bwN))_\G
= 2 \int_0^T (\bQn \pi_n \bH[\bQn], \bW_\G(\bwN))_\G\\
= 2 \int_0^T (\pi_n \bH[\bQn], \bQn \bW_\G(\bwN))_\G
\xrightarrow{n\to\infty} 2 \int_0^T (\bH[\bQ], \bQ \bW_\G(\bwN))_\G\\
\xrightarrow{N\to\infty} 2 \int_0^T (\bH[\bQ], \bQ \bW_\G(\bw))_\G
= \int_0^T (\bSigma_{\Gamma}[\bQ], \nablaG \bw)_\G,
\end{multline}
where in the last limit we have used the convergence $\bW_\G(\bwN) \to \bW_\G(\bw)$ in $\CC([0,T];\bbL^4)$.
By recalling that $\bB$ and $\bnu$ are (time-independent) and bounded, we also obtain via similar arguments that
\begin{equation}\label{momentum-to-limit-4}
\int_0^T (\bB \bSigma_n[\bQn] \bnu, \bwN)_\G
\xrightarrow{n\to\infty} \int_0^T (\bB \bSigma[\bQ] \bnu, \bwN)_\G
\xrightarrow{N\to\infty} \int_0^T (\bB \bSigma[\bQ] \bnu, \bw)_\G.
\end{equation}
Recalling the weak convergence \eqref{H-weak-conv} and noticing that $(\nabla_M \bQn) \, \bwN \xrightarrow{n\to\infty} (\nabla_M \bQ) \, \bwN$ in $\LL^2([0,T];\bbL^2)$ because of \eqref{eq:strong-conv-Q1} and $\bwN \in \CC([0,T];\bH^2) \hookrightarrow \CC([0,T]; \bL^\infty)$, we arrive at
\begin{multline}\label{momentum-to-limit-5}
\int_0^T (\pi_n\bH[\bQn] : \nabla_M \bQn,\bwN)_\G
= \int_0^T (\pi_n\bH[\bQn],(\nabla_M \bQn) \bwN)_\G
\xrightarrow{n\to\infty} \int_0^T (\bH[\bQ],(\nabla_M \bQ) \bwN)_\G\\
\xrightarrow{N\to\infty} \int_0^T (\bH[\bQ],(\nabla_M \bQ) \bw)_\G
= \int_0^T (\bH[\bQ] : \nabla_M \bQ,\bw)_\G,
\end{multline}
Finally, by standard properties of the orthogonal projector $\sP_n$ we obtain
\begin{equation}\label{momentum-to-limit-6}
(\bun_0,\bwN(0))_\G
= (\sP_n \bu_0,\bwN(0))_\G
\xrightarrow{n\to\infty} (\bu_0,\bwN(0))_\G
\xrightarrow{N\to\infty} (\bu_0,\bw(0))_\G.
\end{equation}
Altogether, equations \eqref{momentum-to-limit-ALL} let us pass to the limit in \eqref{integrated-galerkin-momentum} and conclude that $(\bu,\bQ)$ satisfies \eqref{eq:integrated-momentum}.
\end{subequations}

\subsubsection{Step 3: liquid-crystal kinematic equation satisfied by \texorpdfstring{$(\bu,\bQ)$}{(u,Q)}}\label{sec:3step-LCkinematics}
Let $\bPsi \in C^1([0,T];\bbH^1_S)$ be such that $\bPsi(T) = \mathbf{0}$.
For any $t \in [0,T]$, we write
\begin{equation*}
\bPsi(t) = \sum_{i=1}^\infty h_i(t) \bAi,
\end{equation*}
with $h_i(t) = (\bPsi(t),\bAi)_\G$ and the series converging in $C([0,T];\bbH^m)$, $m \in \{0,1\}$ (cf.~\autoref{lemma:divgrad-eigen}).
Notice that $h_i \in C^1([0,T])$, $h_i(T) = 0$ for each $i \in \NN$.
Moreover, it holds that $\partial_t \bPsi(t) = \sum_{i=1}^\infty h_i'(t) \bAi \in C([0,T];\bbH^1_S)$ with the series also converging in the previously mentioned norms.

For a fixed $N \leq n$ define $\bPsiN := \sum_{i=1}^N h_i \bAi$.
Taking $\bPsiN(t) \in \cQ_n$ for each $t \in [0,T]$ in \eqref{galerkin-LCkinematics}, integrating in time (including an integration-by-parts in the first term), and using the initial condition $\bQn(0) = \bQn_0$, we arrive at
\begin{multline}\label{integrated-galerkin-LCkinematics}
- \int_0^T (\bQn, \partial_t \bPsiN)_\G + \int_0^T ((\nabla_M \bQn) \bun, \bPsiN)_\G - \int_0^T (\bS[\bQn,\bun],\bPsiN)_\G\\
= M\int_0^T (\bH[\bQn],\bPsiN)_\G + (\bQn_0,\bPsiN(0)).
\end{multline}
We can pass to the limit in the first term of the left-hand side of \eqref{integrated-galerkin-LCkinematics} because of \eqref{eq:strong-conv-Q1}, whereas for the second term we use the convergences $\nabla_M \bQn \rightharpoonup \nabla_M \bQ$ in $\LL^2([0,T];\LL^4(\G, \RR^{(d+1)^3}))$ (cf.~\eqref{eq:weak-conv-Q2}), \eqref{eq:strong-conv-u1} and $\bPsiN \to \bPsi$ in $\CC([0,T];\bbL^4)$;
namely
\begin{subequations}\label{kinematics-to-limit-ALL}
\begin{align}
\label{kinematics-to-limit-1} & \int_0^T (\bQn, \partial_t \bPsiN)_\G
\xrightarrow{n\to\infty} \int_0^T (\bQ, \partial_t \bPsiN)_\G
\xrightarrow{N\to\infty} \int_0^T (\bQ, \partial_t \bPsi)_\G
,\\
\label{kinematics-to-limit-2} & \int_0^T ((\nabla_M \bQn) \bun, \bPsiN)_\G
\xrightarrow{n\to\infty} \int_0^T ((\nabla_M \bQ) \bu, \bPsiN)_\G
\xrightarrow{N\to\infty} \int_0^T ((\nabla_M \bQ) \bu, \bPsi)_\G.
\end{align}
On the other hand, recalling that $\bS = \bS_\G + \bS_*$ along with the definitions \eqref{S-def} of $\bS_\G$ and $\bS_*$ and the definition  \eqref{starspintensor} of $\bW_*(\bv) = \bB \bv \otimes \bnu - \bnu \otimes \bB \bv$ (with $\bB$ and $\bnu$ uniformly bounded), convergences \eqref{eq:strong-conv-Q1}, \eqref{eq:weak-conv-u1}, \eqref{eq:weak-conv-u2} and \eqref{eq:strong-conv-u1} and the continuous injection $\bbH^1 \hookrightarrow \bbL^4$ imply
\begin{multline}\label{kinematics-to-limit-3}
\int_0^T (\bS_\G[\bQn,\bun],\bPsiN)_\G
= \int_0^T (\bW_\G(\bun) \bQn - \bQn \bW_\G(\bun) ,\bPsiN)_\G\\
\xrightarrow{n\to\infty} \int_0^T (\bW_\G(\bu) \bQ - \bQ \bW_\G(\bu) ,\bPsiN)_\G
\xrightarrow{N\to\infty} \int_0^T (\bS_\G[\bQ,\bu] ,\bPsi)_\G
\end{multline}
and
\begin{multline}\label{kinematics-to-limit-4}
\int_0^T (\bS_*[\bQn,\bun],\bPsiN)_\G
= \int_0^T (\bW_*(\bun) \bQn - \bQn \bW_*(\bun) ,\bPsiN)_\G\\
\xrightarrow{n\to\infty} \int_0^T (\bW_*(\bu) \bQ - \bQ \bW_*(\bu) ,\bPsiN)_\G
\xrightarrow{N\to\infty} \int_0^T (\bS_*[\bQ,\bu] ,\bPsi)_\G.
\end{multline}
Finally, the weak convergence \eqref{H-weak-conv} and classical properties of the orthogonal projector $\pi_n$ enable us to treat the terms in the right-hand side of \eqref{integrated-galerkin-LCkinematics};
more precisely we have
\begin{align}
\label{kinematics-to-limit-5}
\int_0^T (\bH[\bQn],\bPsiN)_\G
\xrightarrow{n\to\infty} \int_0^T (\bH[\bQ],\bPsiN)_\G
\xrightarrow{N\to\infty} \int_0^T (\bH[\bQ],\bPsi)_\G,\\
\label{kinematics-to-limit-6} 
(\bQn_0,\bPsiN(0))_\G
= (\pi_n \bQ_0,\bPsiN(0))_\G
\xrightarrow{n\to\infty} (\bQ_0,\bPsiN(0))_\G
\xrightarrow{N\to\infty} (\bQ_0,\bPsi(0))_\G.
\end{align}
\end{subequations}
In this way, identities \eqref{kinematics-to-limit-ALL} let us pass to the limit in \eqref{integrated-galerkin-LCkinematics} to deduce that $(\bu,\bQ)$ satisfies \eqref{eq:integrated-LCkinematics}.

\subsubsection{Step 4: energy inequality}\label{sec:4step-Eineq}
Recall from \eqref{discrete-apriori} that the Galerkin solution $(\bun,\bQn)$ satisfies the following identity:
\begin{equation}
\tag{\ref{discrete-apriori}}
E_{LdG}[\bQn(t)] + K[\bun(t)] + \int_0^t \Big(2\mu  \|\bD_\G \bun(s) \|^2 + M \|\pi_n \bH[\bQn(s)]\|^2\Big) \dd s
= E_{LdG}[\bQn_0] + K[\bun_0].
\end{equation}
Standard properties of the weak convergence \cite[Proposition 3.5-(iii)]{Brezis2011} applied to \eqref{eq:weak-conv-u2} and \eqref{H-weak-conv} imply that $\forall t \in (0,T)$:
\begin{subequations}\label{energy-to-lim-ALL}

\begin{equation}
\begin{aligned}
& \|\bD_\G \bu\|_{\LL^2([0,t];\bbL^2)} \leq \liminf_{n\to\infty} \|\bD_\G \bun\|_{\LL^2([0,t];\bbL^2)},\\
& \|\bH[\bQ]\|_{\LL^2([0,t];\bbL^2)} \leq \liminf_{n \to \infty} \|\pi_n \bH[\bQn]\|_{\LL^2([0,t];\bbL^2)}.
\end{aligned}
\end{equation}
In turn, applying the a.e.~subsequence convergence property of strong $\LL^2$ convergence \cite[Theorem 4.9]{Brezis2011} to \eqref{eq:strong-conv-Q1} and \eqref{eq:strong-conv-u1} we deduce that, up to a subsequence:
\begin{equation}\label{eq:ae-in-time-conv-uQ}
\bQn(t) \to \bQ(t) \text{ in } \bbH^1, \qquad \bun(t) \to \bu(t) \text{ in } \bL^2, \qquad\qquad \text{a.e.~$t \in (0,T)$}.
\end{equation}

Moreover, since $|F(\bQn)-F(\bQ)| \leq C(1+|\bQ|^3+|\bQn|^3)|\bQn-\bQ|$ because $F$ is a quartic polynomial, integrating over $\G$ and using H\"older's inequality imply
\begin{equation*}
\|F(\bQn)-F(\bQ)\|_{0,1}
\leq C \|1+|\bQ|^3+|\bQn|^3\|_{0,\frac{4}{3}} \|\bQn-\bQ\|_{0,4}
\leq C \left( |\G|^{\frac{3}{4}} + \|\bQ\|_{0,4}^3 + \|\bQn\|_{0,4}^3  \right) \|\bQn-\bQ\|_{0,4}.
\end{equation*}
Consequently, the strong convergence \eqref{eq:strong-conv-Q2} for $p=4$ yields

\begin{equation}
\int_\Gamma F[\bQ(t)] = \lim_{n\to\infty} \int_\G F[\bQn(t)], \qquad \forall t \in (0,T).
\end{equation}
Finally, the strong approximation properties $\pi_n \bQ_0 \rightarrow \bQ_0$ and $\sP_n \bu_0 \rightarrow \bu_0$, together with the continuous injection $\bbH^1 \hookrightarrow \bbL^4$, and the $\bL^2$-continuity of $\sP_n$ yield
\begin{equation}
\begin{aligned}
E_{LdG}[\bQ_0] & = \lim_{n\to\infty} E_{LdG}[\bQn_0],\\
K[\bu_0] & = \lim_{n\to\infty} K[\bun_0].
\end{aligned}
\end{equation}
Collecting identities \eqref{energy-to-lim-ALL}, we can now take $\liminf_{n\to\infty}$ in \eqref{discrete-apriori} to conclude that $(\bu,\bQ)$ satisfies the energy inequality \eqref{eq:continuous-apriori}.
\end{subequations}

\subsection{Proof of \autoref{thm:pressure-recovery}: recovery of the pressure field \texorpdfstring{$\uppi$}{π}}\label{sec:proof-recovery_pressure}
Fix $\varepsilon > 0$ and let $p := (d+\varepsilon)^*$.
Due to the strong nonlinearities of the Beris--Edwards system \eqref{eq:BE-system}, we now abandon the $\LL^2$ setting in space for the strictly weaker $\LL^p$-based setting.
By virtue of \autoref{prop:Helmholtzdecomposition} (Helmholtz--Weyl decomposition on $\G$), each $\bw \in \CC^1([0,T];\bW^{2,p^*}_t)$ can be uniquely decomposed as $\bw = \overline{\bw} + \nabla_\G \phi$, where $\overline{\bw} \in \CC^1([0,T];\bW^{2,p^*}_{t,\sigma})$ and $\phi \in \CC^1([0,T];\WW^{3,p^*}_\#)$.
Consequently, recalling the incompressibility condition $\divG \bu = 0$ and \autoref{thm:existence-weak-solution} (existence of weak solutions), it is necessary and sufficient to show the existence of $\uppi \in \LL^2([0,T];\LL^p_\#)$ (in principle dependent on $\varepsilon$) such that
\begin{multline}\label{int_mom-pressure-gradphi}
- \rho \int_0^T (\bu,\partial_t\nablaG\phi)_\G + \rho \int_0^T ( (\nablaG \bu) \bu, \nablaG\phi )_\G + 2 \mu \int_0^T (\bD_\G \bu, \bD_\G \nablaG\phi)_\G - \rho \int_0^T (\uppi,\DeltaG\phi)_\G\\
+ \int_0^T (\bSigma_{\Gamma}[\bQ], \nablaG \nablaG\phi)_\G + \int_0^T (\bH[\bQ] : \nabla_M \bQ,\nablaG\phi)_\G + 2 \int_0^T (\bB \bSigma[\bQ] \bnu, \nablaG\phi)_\G
= \rho(\bu_0,\nablaG\phi(0))_\G
\end{multline}
for any $\phi \in \CC^1([0,T];\WW^{3,p^*}_\#)$ such that $\nablaG\phi(T) = \mathbf{0}$.
Notice that because of the condition $\divG \bu_0 = 0$, it follows that $(\bu_0, \nablaG\phi(t)) = 0$ for each $t \in [0,T]$;
consequently, we can drop the final-time requirement on $\nablaG\phi$.
It is possible to further simplify \eqref{int_mom-pressure-gradphi}.
Indeed, since time and spatial derivatives commute and because of the incompressibility condition on $\bu$, it follows that
\begin{equation*}
(\bu,\partial_t \nablaG\phi)_\G = (\bu, \nablaG \partial_t\phi)_\G = 0, \qquad \text{a.e.~on $(0,T)$}.
\end{equation*}
In turn, due to the symmetry of the covariant Hessian $\nablaG\nablaG \phi$ (see, e.g.~\cite[Proposition B.2]{BenavidesNochettoShakipov2025-b}), the skew-symmetry of $\bSigma_\G[\bQ]$ and identity \cite[Lemma 2.13]{BenavidesNochettoShakipov2025-b} we have that
\begin{equation*}
\int_0^T (\bD_\G \bu, \bD_\G \nablaG\phi)_\G
= \int_0^T (\bD_\G \bu, \nablaG \nablaG\phi)_\G
= \int_0^T (\nablaG \bu, \nablaG \nablaG\phi)_\G
= -\int_0^T ( (\tr(\bB)\bB - \bB^2)\bu , \nablaG\phi )_\G
\end{equation*}
and
\begin{equation*}
\bSigma_{\Gamma}[\bQ] : \nablaG \nablaG\phi
= 0, \qquad \text{a.e.~on $(0,T)$, a.e.~on $\G$}.
\end{equation*}
In this way, in order to establish the existence and uniqueness of $\pi \in \LL^2([0,T];\LL^p_\#)$ that satisfies \eqref{int_mom-pressure-gradphi} (and so \eqref{eq:integrated-momentum-withpressure} as well), it suffices to establish it for
\begin{equation}\label{int_mom-pressure-gradphi_simplified}
- \rho\int_0^T (\uppi,\DeltaG\phi)_\G
= \int_0^T (\bF[\bu,\bQ], \nablaG\phi)_\G
, \qquad \forall \phi \in \CC^1([0,T];\WW^{3,p^*}_\#),
\end{equation}
where $\bF[\bu,\bQ] := -\rho\,(\nablaG \bu) \bu + 2\mu\,(\tr(\bB)\bB - \bB^2)\bu - \bH[\bQ] : \nabla_M \bQ - 2\bB \bSigma[\bQ] \bnu$.
We claim that $\bF[\bu,\bQ] \in \LL^2([0,T];(\bW^{1,d+\varepsilon}_t)')\cong (\LL^2([0,T];\bW^{1,d+\varepsilon}_t))'$ for every $\varepsilon \in (0,1)$.
In fact, for every $\bv \in \LL^2([0,T];\bW^{1,d+\varepsilon}_t)$, H\"older's inequalities, Sobolev embeddings $\bW^{1,d+\varepsilon} \hookrightarrow \bL^\infty \hookrightarrow \bL^3$ and $\bbH^1 \hookrightarrow \bbL^6$, and interpolation inequalities \eqref{lady-2D} for $d=2$ and \eqref{lady-3D} for $d=3$, imply that
\begin{subequations}\label{estimate-F}
\begin{multline}
\int_0^T \left| ((\nablaG \bu)\bu, \bv)_\G  \right|
= \int_0^T \left| (\bu \otimes \bu, \nablaG \bv)  \right|
\leq \|\bu\otimes\bu\|_{\LL^2([0,T];\bbL^{d^*})} \|\nablaG\bv\|_{\LL^2([0,T];\bbL^d)}\\
\leq \|\bu\|_{\LL^4([0,T];\bL^{2d^*})}^2 \|\bv\|_{\LL^2([0,T];\bW^{1,d})}
\leq C \|\bu\|_{\LL^2([0,T];\bH^1)}\|\bu\|_{\LL^\infty([0,T];\bL^2)} \|\bv\|_{\LL^2([0,T];\bW^{1,d+\varepsilon})},
\end{multline}
\begin{equation}
\int_0^T \left| ((\tr(\bB)\bB - \bB^2)\bu,\bv)_\G  \right|
\lesssim \|\bu\|_{\LL^2([0,T];\bL^2)} \|\bv\|_{\LL^2([0,T];\bL^2)}
\leq C \|\bu\|_{\LL^2([0,T];\bL^2)} \|\bv\|_{\LL^2([0,T];\bW^{1,d+\varepsilon})},
\end{equation}
\begin{multline}
\int_0^T \left| (\bH[\bQ]:\nabla_M \bQ ,\bv)_\G  \right|
\leq \|\bH[\bQ]\|_{\LL^2([0,T];\bbL^2)} \|\nabla_M \bQ\|_{\LL^\infty([0,T];\bbL^2)} \|\bv\|_{\LL^2([0,T];\bL^\infty)}\\
\leq C_\varepsilon \|\bH[\bQ]\|_{\LL^2([0,T];\bbL^2)} \|\nabla_M \bQ\|_{\LL^\infty([0,T];\bbL^2)} \|\bv\|_{\LL^2([0,T];\bW^{1,d+\varepsilon})}
\end{multline}
and
\begin{multline}
\int_0^T \left| (\bB \bSigma[\bQ] \bnu, \bv)_\G \right|
\stackrel{\eqref{Sigma-def}}{\lesssim} \|\bQ\|_{\LL^\infty([0,T];\bbL^6)} \|\bH[\bQ]\|_{\LL^2([0,T];\bbL^2)} \|\bv\|_{\LL^2([0,T];\bL^3)}\\
\leq C \|\bQ\|_{\LL^\infty([0,T];\bbH^1)} \|\bH[\bQ]\|_{\LL^2([0,T];\bbL^2)} \|\bv\|_{\LL^2([0,T];\bW^{1,d+\varepsilon})}.
\end{multline}
\end{subequations}
Altogether, and considering the regularity of $(\bu,\bQ)$, we deduce from \eqref{estimate-F} that $\bF[\bu,\bQ] \in \LL^2([0,T];(\bW^{1,d+\varepsilon}_t)')$, as previously stated.
Now, observe that for a.e.~$t \in (0,T)$, $\bF[\bu,\bQ]\big\rvert_t \in (\bW^{1,d+\varepsilon}_t)'$.
Consequently, a straightforward combination of the well-posedness of the ultra-weak formulation of the stationary Laplace--Beltrami operator $\DeltaG$ \cite[Lemma 4.3]{BenavidesNochettoShakipov2025-a} and the Banach--Ne\v{c}as--Babu\v{s}ka theorem \cite[Theorem 2.6]{ErnGuermond2004} yields the existence of a unique $\uppi \in \LL^2([0,T];\LL^{(d+\varepsilon)^*}_\#)$ such that
\begin{equation}\label{eq:problem-for-pressure}
- \rho\int_0^T (\uppi,\DeltaG\phi)_\G
= \int_0^T \langle \bF[\bu,\bQ], \nablaG\phi \rangle_{(\bW^{1,d+\varepsilon}_t)' \times \bW^{1,d+\varepsilon}_t}
= \int_0^T (\bF[\bu,\bQ], \nablaG\phi)_\G
, \qquad \forall \phi \in \LL^2([0,T];\WW^{2,d+\varepsilon}_\#).
\end{equation}
In particular, $\uppi$ satisfies \eqref{int_mom-pressure-gradphi_simplified} with $p = (d+\varepsilon)^*$ (and therefore \eqref{int_mom-pressure-gradphi} and \eqref{eq:integrated-momentum-withpressure} as well).
Moreover, $\uppi$ satisfies the estimate
\begin{equation*}
\rho \|\uppi\|_{\LL^2([0,T];\LL^{(d+\varepsilon)^*}_\#)}
\leq \wt C_\varepsilon \|\bF[\bu,\bQ]\|_{\LL^2([0,T];(\bW^{1,d+\varepsilon}_t)')},\\
\end{equation*}
which, combined with \eqref{estimate-F}, yields \eqref{eq:pressure-apriori}.

Finally, the existence of a pressure field $\uppi$ independent of $\varepsilon$ follows from noticing that
\begin{align*}
\LL^2([0,T];\LL^{(d+\varepsilon_1)^*}_\#) & \subseteq \LL^2([0,T];\LL^{(d+\varepsilon_2)^*}_\#),\\
\LL^2([0,T];\WW^{2,d+\varepsilon_1}_\#) & \supseteq \LL^2([0,T];\WW^{2,d+\varepsilon_2}_\#)
\end{align*} if $0 < \varepsilon_1 < \varepsilon_2$, and recalling the existence (for a given $\varepsilon$) of a unique solution of problem \eqref{eq:problem-for-pressure}.

\section*{Acknowledgements}
All three authors were partially supported by the NSF Grants DMS-1908267 and DMS-2512392.
We thank the anonymous referee for their constructive suggestions.

The authors wrote and developed the entire original manuscript, with AI assistance limited to consulting Google Gemini to clarify the justification for differentiating the eigenfunction expansion in \autoref{sec:2step-momentum} and \autoref{sec:3step-LCkinematics}.
During revision, we used ChatGPT (OpenAI) to identify errors and inconsistencies.
The authors verified the resulting changes and take full responsibility for the manuscript.

\bibliographystyle{abbrv}
\bibliography{references}

\appendix

\section{Formal derivation of the weak formulation}\label{sec:formal-wf}

In this section we utilize integration-by-parts formulas \eqref{eq:int-by-parts}, \eqref{covariantIBP-matrices}, \eqref{externalIBP-matrices} to formally derive a continuous-in-time variational formulation for the surface Beris--Edwards system \eqref{eq:BE-system} that inspires \autoref{def:weak-solution} (weak solution).
These calculations, albeit being relatively simple, are not present in \cite{BouckNochettoYushutin2024} and therefore we include them in this paper for the sake of completeness.

In what follows all functions involved are assumed to be sufficiently smooth.
Assume that $\bu$ is tangential to $\Gamma$ and that it satisfies the incompressibility condition \eqref{incompressibility}.
Then, for any tangential and $\divG$-free function $\bv$, we have
\begin{equation*}
(\bP \divG \bD_\G \bu, \bv)_\G
= (\divG \bD_\G \bu, \bv)_\G
\stackrel{\eqref{covariantIBP-matrices}}{=} - (\bD_\G \bu, \bD_M \bv)_\G + ( \tr(\bB) \, \bD_\G (\bu) \bnu, \bv )_\G.
\end{equation*}
Because of the identity $\bD_\G \bu = \bP \bD_M(\bu) \bP$ it holds that $\bD_\G \bu :  \bD_M \bv = \bD_\G \bu : \bD_\G \bv$ and $\bD_\G (\bu) \bnu = \mathbf{0}$, whence we finally obtain
\begin{equation}\label{proto-weakformulation-1}
(\bP \divG \bD_\G \bu, \bv)_\G = - (\bD_\G \bu, \bD_\G \bv)_\G.
\end{equation}
By testing $\bf_\G + \bf_*$ with $\bv$ we can write
\begin{equation*}
(\bf_\G,\bv)_\G + (\bf_*,\bv)_\G
= (- \bP \divG \bSigma_\G, \bv)_\G + (\bH : \nabla_M \bQ,\bv)_\G + 2 (\bB \bSigma \bnu, \bv)_\G.
\end{equation*}
The first term in the right-hand-side can be rewritten upon utilizing integration-by-parts, and by noticing that $\bSigma_\G: \nabla_M \bv = \bSigma_\G: \nablaG \bv$ and $\bSigma_\G \bnu = \mathbf{0}$, because $\bSigma_\G = \bP \bSigma \bP$, obtaining
\begin{equation*}
(- \bP \divG \bSigma_\G, \bv)_\G
= (- \divG \bSigma_\G, \bv)_\G
\stackrel{\eqref{covariantIBP-matrices}}{=} (\bSigma_\G, \nabla_M \bv)_\G - (\tr(\bB) \bSigma_\G \bnu, \bv)_\G
= (\bSigma_\G, \nablaG \bv)_\G.
\end{equation*}
Thus,
\begin{equation}\label{proto-weakformulation-2}
(\bf_\G,\bv)_\G + (\bf_*,\bv)_\G = (\bSigma_\G, \nablaG \bv)_\G + (\bH : \nabla_M \bQ,\bv)_\G + 2 (\bB \bSigma \bnu, \bv)_\G.
\end{equation}
By collecting \eqref{proto-weakformulation-1} and \eqref{proto-weakformulation-2}, noticing that $(\nablaG \uppi, \bv)_\G = 0$ for $\bv$ $\divG$-free, and by integration-by-parts formula \eqref{eq:int-by-parts} we arrive at the following weak form of the momentum equation \eqref{momentum-eq}:
\begin{equation}\label{continuous-weakmomentum}
\rho\Big\{ (\partial_t \bu, \bv)_\G
+ \left( (\nablaG \bu)\bu, \bv \right)_\G \Big\}
+ 2\mu (\bD_\G \bu, \bD_\G \bv)_\G
= -(\bSigma_\G[\bQ], \nablaG \bv)_\G - (\bH[\bQ] : \nabla_M \bQ,\bv)_\G - 2 (\bB \bSigma[\bQ] \bnu, \bv)_\G.
\end{equation}
For any $\bT: \Gamma \rightarrow \RR^{(d+1)\times(d+1)}$, we weakly impose the liquid-crystal kinematic equation \eqref{LC-kinematics} as
\begin{equation}\label{continuous-weakkinematics}
(\partial_t \bQ, \bT)_\G + ((\nabla_M \bQ) \bu, \bT)_\G = M (\bH[\bQ],\bT)_\G + (\bS[\bQ,\bu],\bT)_\G.
\end{equation}
By integrating \eqref{continuous-weakmomentum} and \eqref{continuous-weakkinematics} in time, and integrating by parts so as to transfer the time derivatives to the test functions $\bv$ and $\bT$, we respectively arrive at the formal analogues of equations \eqref{eq:integrated-momentum} and \eqref{eq:integrated-LCkinematics} in \autoref{def:weak-solution}.
Notice that, should we not restrict $\bv$ to satisfy $\divG \bv = 0$, a term involving the pressure field $\uppi$ would be present in \eqref{continuous-weakmomentum}, which after being integrated in time, would lead to the formal analogue of \eqref{eq:integrated-momentum-withpressure} in \autoref{thm:pressure-recovery} (recovery of the pressure field).
\end{document}